\documentclass[11pt]{amsart}
\usepackage[margin=1in,foot=30pt]{geometry}
\usepackage{amsmath,amsthm,amssymb}
\usepackage{hyperref}

\numberwithin{equation}{section}
\newtheorem{prop}{Proposition}
\newtheorem{lemma}[prop]{Lemma}

\newtheorem{thm}[prop]{Theorem}
\newtheorem{cor}[prop]{Corollary}

\numberwithin{prop}{section}
\theoremstyle{definition}
\newtheorem{defn}[prop]{Definition}

\newtheorem{rmk}[prop]{Remark}

\usepackage{xcolor}

\newcommand{\del}{\partial}

\newcommand{\brs}[1]{\left| #1 \right|}
\newcommand{\brk}[1]{\left[ #1 \right]}
\newcommand{\prs}[1]{\left( #1 \right)}

\newcommand{\hsp}{\hspace{0.5cm}}
\newcommand{\N}{\nabla}

\newcommand{\lap}{\Delta}

\renewcommand{\ge}{\epsilon}

\newcommand{\vp}{\varphi}

\newcommand{\AR}{\mathbb{R}}
\newcommand{\Sol}{\operatorname{Sol}}

\DeclareMathOperator{\cB}{B}

\begin{document}
\title{A Monotonicity Formula for the Extrinsic Biharmonic Map Heat Flow}
\author{Elena M\"{a}der-Baumdicker}
\address{Freie Universit\"at Berlin, Institute of Mathematics, Arnimallee 3, D-14195 Berlin, Germany}
\email{elena.maeder-baumdicker@fu-berlin.de}
\author{Nils Neumann}
\address{Technical University of Darmstadt, Department of Mathematics, Schlossgartenstr.~7, D-64289 Darmstadt, Germany}
\textbf{\email{neumann@mathematik.tu-darmstadt.de}}
\dedicatory{This article is dedicated to Richard S.\ Hamilton, in honor of his \\ profound and enduring contributions to geometric analysis.}
\begin{abstract} We investigate new properties of the biharmonic heat kernel on Euclidean space and derive a localized entropy-like quantity $\Psi$ for the extrinsic biharmonic map heat flow $u:\mathbb R^n\times[0,T) \to N$, where $N$ is a smooth closed manifold or $N=\AR^k$.  For general $n$, we derive a formula for the $R$-derivative of $\Psi$. Here, $R$ denotes (sometimes up to a constant) the parabolic time-distances to a fixed $t_0\in (0,T]$ via $t=t_0- R^4$. For $1\leq n\leq 3$ we show that $\Psi$ is monotone in the sense $\Psi(R_1) \leq \Psi(R_2) + c \cdot (R_2^\beta - R_1^\beta)$ for $R_1 < R_2$ and a positive $\beta>0$. Here, $c$ only depends on $u_0$, $n$, $N$ and $T$.  For $n= 4$ we show a similar statement, namely $\Psi(R_1) \leq \Psi(R_2) + c \cdot (\log(R_2) - \log(R_1))$ for $R_1 < R_2$. The case $n=4$ requires an additional smallness condition on $u_0$.
\end{abstract}
\maketitle
\section{Introduction}
Parabolic monotonicity formulas are a fundamental technique for second order geometric evolution equations such as Huisken's monotonicity formula for mean curvature flow \cite{HuiskenMono}, the monotonicity formulas by Struwe and Chen-Struwe for the harmonic map heat flow \cite{StruweMono, ChenStruwe}, the one for Yang-Mills heat flow cf.\ Hamilton \cite{Hamilton}, Chen-Shen \cite{CS1}, Hong-Tian \cite{HongTian2004}, and several ones for the Ricci flow shown by Perelman \cite{Perelman} and Hamilton \cite{HamRicci}. Even for flows with a global constraint such as the volume preserving mean curvature flow, a parabolic monotonicity formula is available as established by the first author in \cite{MeinPaper, VPMCF}. In several cases the principle to derive such a monotonicity formula is similar: Consider the solution of a gradient descent flow $u$ for an energy, then weight the energy with the (backwards) heat kernel on the base space (or an appropriate rescaling of the ambient backwards heat kernel in the case of the mean curvature flow) and multiply with the correct power of $(T-t)$ such that the object is invariant under parabolic rescaling. We denote this weighted energy by $\Phi$. Then it is shown that the time derivative of $\Phi$ is non-positive. Another way of expressing this is that the derivative with respect to $R$ is non-negative, where $t= t_0 -R^2$ in the second order case.  For localization purposes a suitable cut-off function $\varphi$ has been introduced in the integrals \cite{ChenStruwe, HongTian2004, CS1} so that a well-controlled error term appears in the monotonicity inequality in the way
\begin{align}\label{ineq:mono}
\Phi(u, R_1, \varphi) \leq C \exp(C(R_2^p- R_1^p))\,\Phi(u, R_2, \varphi) + C(R_2^q-R_1^q)\, E(0),
\end{align}
for a constant $C>0 $, for some $p,q>0$ and for all $R_1<R_2$. Usually, $E(0)$ is the initial energy. On general manifolds $M$, Hamilton showed that the harmonic map heat flow satisfies a similar inequality \cite{Hamilton}. Such an inequality is also called a \emph{monotonicity formula} although $\Phi$ is not monotone in the classical strong sense.
\\

To the authors' knowledge, no parabolic monotonicity formula has been known so far for fourth-order geometric evolution equations. In this article we present such a monotonicity formula for the scalar biheat equation on $\AR^n$ and for the extrinsic biharmonic map heat flow on $\AR^n$. The latter is a parabolic system of semilinear equations. We first recall that there are at least two notions of biharmonic maps. We consider maps 
\[f:(M^n,g) \to (N,h)\hookrightarrow \AR^k\]
between two Riemannian manifolds, where we assume $N$ to be isometrically embedded in $\AR^k$. The \emph{intrinsic biharmonic energy} of such an $f$ is  
\[ E_1(f): = \frac{1}{2} \int_M |\tau(f)|^2 d\mu_g,
\]
where $\tau(f)$ is the \emph{tension field} of $f$, that is, the negative $L^2$-gradient of the intrinsic Dirichlet energy of $f$. The tension field can be computed as $\tau(f) = \Delta f + A(f)(\nabla f,\nabla f)$ when seeing $f$ as a map to $\AR^k$. The leading term $\Delta f$ of $\tau(f)$ is the Laplace-Beltrami operator of $f:M \to N\hookrightarrow\AR^k$ when seeing $f$ as a map into $\AR^k$.
The \emph{extrinsic biharmonic energy} is 
\[
 E_2(f) := \frac{1}{2}  \int_M |\Delta f|^2 d\mu_g,
\]
where $\Delta f $ is the leading term from $\tau(f)$. Both energies come with their own challenges and are of significant interest which is, for example, explained in the work of Moser and Hornung-Moser \cite{Moser2008, Moser2, Moser4}. The associated $L^2$-gradient flows of 
$E_1$  and $E_2$ are known as the \emph{intrinsic and extrinsic biharmonic map heat flow}, respectively. Their behavior remains largely unexplored in general settings. We give an overview of the most importants results that had been obtained so far without claiming that this is a complete list.
 For the extrinsic case, Lamm \cite{LammExtr} proved short-time existence and, in fact, long-time existence and smooth subconvergence to an extrinsic biharmonic map for the case $n \leq  3$. For the critical dimension 
$n=4$ he also established long-time existence and subconvergence, provided that the initial energy is small \cite{LammExtr}.
  Gastel \cite{Gastelextr} and independently Wang \cite{WangBiheat} showed long-time-existence of weak solutions for the critical dimension $n = 4$ which are smooth except for finitely many isolated times. Huang etal.\ \cite{extrBiharmBrdy} proved a similar result for the boundary case in $n=4$. Liu and Yin presented examples demonstrating the occurrence of finite-time singularities in the critical dimension $n=4$ \cite{LiuYin}. A similar result was obtained by Cooper in \cite{Cooper}. Well-posedness for small times for rough initial data for the extrinsic and intrinsic case was proven by Wang \cite{WangBiheat2}. \\
This leads us to the intrinsic biharmonic map heat flow, the flow corresponding to $E_1$. Short-time existence follows from the work of Lamm  \cite{LammExtr}, and he also proved in \cite{Lamm2005} that if the sectional curvature of the ambient space 
$N$ is non-positive, then the flow does not develop singularities if $n\leq 4$, and it subconverges smoothly to a harmonic map. Moser started the analysis of possible singularities of the intrinsic flow \cite{Moser3}. Aside from these results in specific cases, our understanding of the intrinsic flow remains quite limited. \\

Following the proofs of the monotonicity formula for the harmonic map heat flow \cite{StruweMono, Hamilton} there are two equivalent ways to present the computations. Either one defines the weighted energy over the base manifold $M$ and takes the time-derivative as it was presented by Hamilton in \cite{Hamilton}. Or one works on slices or horizontal layers which was done by Struwe in \cite{StruweMono}. We decided to follow the latter approach. For that, we consider a solution of the biharmonic map heat flow $u:\AR^n\times[0,T)\to N\hookrightarrow\AR^k$, $\partial_t u= -\Delta^2 u - f(u)$, where $f(u)$ is a nonlinear term we recall in Section~\ref{sect:variations} and let $(x_0,t_0)\in \AR^n\times(0,T)$. Then we define for $0<R<\frac{(t_0)^{\frac{1}{4}}}{2}$
\begin{align*}
T_R(t_0)&:= \AR^n \times \{t: t_0 - 16 R^4 < t < t_0-R^4\},\\
S_R(t_0)&:= \AR^n \times\{t = t_0-R^4\}.
\end{align*}
We weight the energy $E_2$ with the kernel of the (backwards) biheat equation $\partial_t B = \Delta^2 B$ on $\AR^n$ concentrating at $(x_0,t_0)$,  which is (see Galaktionov-Pohozaev \cite{Galaktionov})
\begin{align*}
 B(x,t):= B_{(x_0,t_0)}(x,t) &= \frac{\alpha_n}{(t_0-t)^{n/4}} f_n\left(\frac{|x-x_0|}{(t_0-t)^{1/4}}\right), \\
 f_n(\eta) &= \eta^{1-n} \int_0^\infty \exp\left(-s^4\right) (\eta s)^{n/2} \operatorname{J}_{\frac{n-2}{2}} (\eta s) \,ds,
\end{align*}
where $ \operatorname{J}_\nu$ denotes the $\nu$-th Besselfunction of first kind and $\alpha_n>0$ is a normalization constant. For important properties of the biharmonic map heat flow we recommend the work of  Ferrero-Gazzola-Grunau \cite{AGG}, Gazzola-Grunau \cite{GG} and Grunau-Miyake-Okabe \cite{GMO}.
Using \cite{AGG} we will prove a very important property of the kernel in Section~\ref{sect:propbiheat}, see Lemma~\ref{lem:propbiheat}:
\begin{lemma}
    The backwards biharmonic heat kernel $B$ on $\AR^n$ satisfies 
    \[\nabla\Delta B(x,t) = \frac{x-x_0}{4(t_0-t)}B(x,t) \text{ on } \AR^n\times (-\infty, t_0).
    \]
\end{lemma}
This equation has an interesting consequence, namely that the biheat equation on $\AR^n$ satisfies a differential Matrix-Harnack type equation, see Corollary~\ref{cor:diffHarnack}. This may be of independent interest. For the second order case, Li-Yau showed versions of the differentiable Harnack inequality for heat equations in a geometric setting in \cite{LiYau}. In \cite{Hamilton_Matrix}, Hamilton showed differential Matrix Harnack inequality version.
\begin{lemma}
Let $b(x,t) =  \frac{\alpha_n}{t^{n/4}} f_n\left(\frac{|x|}{t^{1/4}}\right)$ with $\alpha_n$ and $f_n$ as above be the (forward) biheat equation concentrating at $(x_0,t_0)=(0,0)$, then $b$ satisfies
\[
b\cdot \nabla_i\nabla_j \Delta b - \nabla_i\Delta b\,\nabla_j b - \frac{\delta_{ij} b^2}{4t} = 0 \qquad \text{ on }\AR^n\times (0,\infty).
\]
\end{lemma}

To address the oscillatory behavior of the biharmonic heat kernel we need to include a cut-off function $\varphi\in C^\infty_c (\AR^n \times[0,T))$ into our computations in analogon to other localized cases \cite{ChenStruwe, CS1}. We define 
\begin{align*}
\Psi(u,R, \varphi)&: = \frac{1}{2}\int_{T_R(t_0)} |\lap u|^2 \, B_{(x_0,t_0)} \varphi\, dxdt - \frac{1}{2}\int_{T_R(t_0)}|\nabla u|^2 \,\Delta B_{(x_0,t_0)} \varphi\, dx dt,\\
\Phi \prs{u, R ,\vp} &:= \tfrac{R^4}{2} \int_{S_R(t_0)} \brs{\lap u}^2  B_{(x_0,t_0)} \vp\,dx- \tfrac{R^4}{2} \int_{S_R(t_0)}\brs{\nabla u}^2 \lap B_{(x_0,t_0)}\vp \,dx.
\end{align*}
Our key computation on $\Psi$ (w.l.o.g.\ $(x_0,t_0)=(0,0)$) shows the following:
\begin{align*}
	\frac{\partial}{\partial R} \Psi(u,R, 1)  &=  \frac{1}{4R} \int_{T_R} \frac{|\nabla u\cdot x + 4 t \partial_t u|^2}{|t|} B \, dxdt \\
						  &\quad + \frac{2}{R} \int_{T_R} [\nabla_i (\Delta u\nabla_i B)]\cdot  (\nabla u\cdot x + 4 t \partial_t u) dxdt.
\end{align*}
This implies in particular that solutions of the biharmonic map heat flow satisfying $\nabla u\cdot x + 4 t \partial_t u=0$ for $t< 0$ are critical points of $\Psi$ with respect to $R$. Solutions (for $t<0$) satisfying this equation are called \emph{shrinking solitons} or \emph{self-shrinker}. Based on these computations and on properties of the biharmonic heat kernel the main results of this article are the following:

\begin{thm}\label{thm1}
Let $u:\AR^n\times[0,T)\to N\hookrightarrow \AR^k$ be a smooth solution of the extrinsic biharmonic map heat flow. Then the quantity $\Psi(u,R,\varphi)$ has the following properties:
\begin{enumerate}
    \item With the notation $u_R(x,t): = u(Rx, R^4t)$ and analogously for $\varphi_R$, the quantity $\Psi (u, R, \varphi)$ is invariant under parabolic rescaling; that is,
\begin{align*}
\Psi(u,R,\varphi) = \Psi(u_R, 1, \varphi_R).
\end{align*}
\item Furthermore, shrinking solitons, i.e.\ solutions $u$ of the biharmonic map heat flow satisfying $\frac{d}{dR}u_R\big|_{R=1}=0$, or equivalently  $\nabla u\cdot (x-x_0) + 4 (t- t_0)\partial_t u=0$ for $t<t_0$, where \\ $(x_0,t_0)\in \mathbb R^n \times (0,T]$, 
are critical points of $\Psi$. That is, $\frac{d}{dR}\Psi(\hat u, R,1)=0$ for shrinking solitons $\hat u$.
\item 
Let $(x_0,t_0)\in \mathbb R^n \times (0,T]$ be fixed. Then there is a smooth cut-off function $\varphi \in C^\infty_c (\AR^n \times[0,T))$ depending on $f_n$ and a constant $C=C(f_n)$ such that, for $0< R\leq  t_0^{1/4}/2$, we have that both terms in the definition of $\Psi$ are non-negative and
      \begin{align*} 
\tfrac{\del}{\del R} \brk{\Psi \prs{u,R,\vp}} 
& \geq \tfrac{1}{2}\tfrac{1}{ R} \int_{T_R(t_0)} \tfrac{ \brs{\nabla u \cdot (x-x_0) + 4 (t-t_0) \partial_t u}^2}{4 \brs{t_0-t}} B \vp \,dx\,dt \\
&\hsp - C  R^{5-n} \sup_{t\in[0, t_0-R_0^4]}\int_{[\varphi>0]}|\N\lap u|^2 \,dx\\
&\hsp - C R^{1-n} \sup_{t\in[0, t_0-R_0^4]} \int_{[\varphi>0]}|\N u|^2 (\cdot, t) \,dx  \\
&\hsp - C R^{3-n} \int_{\AR^n}|\lap u_0|^2 \,dx 
\end{align*}
\end{enumerate}
\end{thm}
We prove similar results for $\Phi$.
The above results hold in all dimensions, but the polynomials in $R$ have negative exponents for many dimensions $n$. We improve this result for subcritical dimensions. First we state a corollary which uses on work of Lamm~\cite{LammExtr}.

 \begin{cor} 
   Let $u:\AR^n \times [0,T] \to N\hookrightarrow \AR^k$ be a smooth solution of the extrinsic biharmonic map heat flow on $\AR^n$ with smooth initial datum $u_0:\AR^n \to N \hookrightarrow \AR^k$, $1\leq n\leq 3$.
   Then, for any $(x_0,t_0)\in\AR^n \times (0,T]$ there is a smooth cut-off function $\varphi \in C^\infty_c (\AR^n \times[0,T))$ depending on $f_n$ and a constant $C>0$ depending on $f_n$, $N$, $n$, such that for  \\ $0< R_1\leq R_2 \leq  \min\{t_0^{1/4}/2,1\}$, we have
   \begin{align*}
   \Psi(u, R_1, \varphi) & \leq \Psi(u, R_2, \varphi) + C \left(R_2^{\frac{4- n}{2}}- R_1^{\frac{4-n}{2}}\right)\tilde E_0 \\
   & + C(R_2^{4-n}-R_1^{4-n}) \int_{\AR^n}|\Delta u_0|^2,
   \end{align*}
 where $\tilde E_0$ depends on $u_0,n$, $N$ and $T$. This implies that $\Psi$ is monotone in a generalized sense for $1\leq n\leq3$.
 \end{cor}

The following corollary deals with the case $n=4$. It requires a strong smallness condition on $u_0$. In the proof, we also use a result from Lamm \cite{LammExtr}.

\begin{cor}\label{cor:last}
 In the case $n=4$, let $u:\AR^4 \times [0,T) \to N\hookrightarrow \AR^k$, be a smooth solution of the extrinsic biharmonic map heat flow on $\AR^4$ with smooth initial datum $u_0:\AR^4 \to N \hookrightarrow \AR^k$. Then there is an $\epsilon_0>0$ and an $r>0$ such that if 
 \[
 \sup_{(x,t)\in \AR^4\times [0,T)} \left[ 2\int_{B_{2r}(x)}|\Delta u|^2 dx +\left(\int_{B_{2r}(x)}|\nabla u|^4 dx\right)^{\frac{1}{2}}\right] <\epsilon_0,
 \]
   then, for any $(x_0,t_0)\in\AR^n \times (0,T]$ there is a smooth cut-off function \\ $\varphi \in C^\infty_c (\AR^n \times [0,T))$ depending on  $f_4$ and a constant $C$ depending on $f_4$ such that, for  $0< R_1<R_2 < \min\{t_0^{1/4}/2,1\}$ we have
   \begin{align*}
   \Psi(u, R_1, \varphi) & \leq \Psi(u, R_2, \varphi) + C (R_2- R_1)\tilde E_0\\
   &\hsp + C (\log(R_2)- \log(R_1))\int|\Delta u_0|^2  + C (\log(R_2)-\log(R_1)) \tilde E_1,
   \end{align*}
   where $\tilde E_0, \tilde E_1$ depend on $u_0$, $N$ and $T$.
\end{cor}

In Section~\ref{sect:variations}, we recall the Euler-Lagrange equation for the energy $E_2$. Properties of the biharmonic heat kernel are derived in Section~\ref{sect:propbiheat}. Then we prove the above monotonicity formulas in Section~\ref{sect:monotonicity}, first in the linear and then in the non-linear case.

\section*{Acknowledgement}
We would like to thank Andrea Malchiodi, Michael Struwe, Jonas Hirsch, Moritz Egert and particularly Tobias Lamm for useful discussions on the topic. We are especially grateful to Tobias Lamm for pointing out a sketch of the proof of Proposition~\ref{prop:absch}. We also thank Luke Peachey and Alessandro Bertellotti for pointing out an error in an earlier version of this manuscript and for extremly helpful suggestions regarding its correction. Furthermore, the first author is deeply grateful to Casey Lynn Kelleher for raising questions about monotonicity formulas for higher-order flows and for the initial calculations they carried out together on this topic in Princeton. The first author is partially supported by the Deutsche Forschungsgemeinschaft (DFG) under the project number MA 7559/1-2.

\section{Variation of the biharmonic energy} \label{sect:variations}

\begin{defn}
For a closed Riemannian manifold $(N,h)$ which is isometrically embedded in $\AR^k$ or for $N=\AR^k$ and for a map $u:\AR^n\to N \hookrightarrow \AR^k$ of regularity $W^{2,2}_{loc}(\AR^n, \AR^k)$ with values in $N$ for almost every $x\in \AR^n$, the extrinsic biharmonic energy is defined as 
\begin{align*}
E_2(u) := \frac{1}{2} \int_{\AR^n} |\Delta u|^2 dx,
\end{align*}
where $\Delta u$ is the Laplacian on $\AR^n$. Critical points of $E$ under compactly supported smooth variations with values in $N$ are called \emph{extrinsic biharmonic maps}.
\end{defn}
The Euler-Lagrange equation of $E_2$ for $N=\mathbb S^{k-1}$ was computed in \cite{ChangWangYang}. Formulations for general $N$ can be found in \cite{LammExtr, WangR4, WangRm}. We will use the following version.

\begin{lemma}[\cite{LammExtr}, Section 2 in \cite{WangRm}]
A smooth extrinsic biharmonic map $u:\AR^n\to N\hookrightarrow \AR^k$ satisfies
\begin{align}
\Delta^2 u \perp T_u N,  \ \ \text{i.e. } \int_{\AR^n}\langle \Delta^2 u, \phi\rangle dx = 0 
\end{align}
for all $\phi\in C^\infty_0(\AR^n,\AR^k) $ with $\phi(x) \in T_{u(x)} N$. We state an equivalent formulation. A smooth extrinsic biharmonic map satisfies
\begin{align} \label{equ:flow}
\Delta^2 u &= - \sum_{i=n+1}^k \bigl(\Delta \langle \nabla u, (d\nu_i\circ u) \nabla u\rangle + \operatorname{div}( \langle \Delta u, (d\nu_i\circ u)\nabla u\rangle)\\
\notag 	   &\qquad \qquad + \langle \nabla \Delta u, (d\nu_i\circ u) \nabla u\rangle \bigr) \nu_i\circ u\\
\notag  &=: - f(u).
\end{align}
where $\{\nu_i\}_{i=n+1}^k$ is a smooth orthonormal frame of the normal space of $N$ near $u(x)$.
\end{lemma}
We follow \cite{LammExtr} to define the (formally) associated $L^2$-gradient flow.
\begin{defn}
A smooth solution $u:\AR^n \times [0,T) \to N \hookrightarrow\AR^k$ satisfying
\begin{align*}
\partial_t u = -\Delta^2 u - f(u) \  \text{ where } f(u) \text{ is defined in } (\ref{equ:flow})
\end{align*}
is called a solution of the \emph{biharmonic map heat flow}.\\
If $N=\AR^k$, then the system of equation decouples and it satisfies to consider a solution $u:\AR^n\times[0,T)\to\AR$ of the scalar biheat equation $\partial_t u=-\Delta^2 u$. In this case, the maximal time of existence is always infinity.
\end{defn}

\begin{rmk} \label{rmk:ortho}
The above definitions and properties particularly imply that $\int_{\AR^n} \langle f(u), \phi \rangle =0 $ for $\phi\in C^\infty_0(\AR^n,\AR^k) $ with $\phi(x) \in T_{u(x)} N$. This means that $f(u)$ is $L^2$-orthogonal to $T_u N$.
\end{rmk}

\section{Properties of the biharmonic heat kernel} \label{sect:propbiheat}
\noindent 
We repeat that the biharmonic heat kernel with concentration at \break $(x_0,t_0)=(0,0)$ of $\Phi \mapsto \partial_t \Phi + \Delta^2\Phi$ is given by \cite{AGG, GG, GMO}
\begin{align}
 b(x,t) &= \alpha_n \frac{f_n(\eta)}{t^{n/4}}, \ \ \ \eta=\frac{|x|}{t^{1/4}},\\
 f_n(\eta) &= \eta^{1-n} \int_0^\infty \exp\left(-s^4\right) (\eta s)^{n/2} \operatorname{J}_{\frac{n-2}{2}} (\eta s) \,ds, \label{equ:integral}
\end{align}
where $ \operatorname{J}_\nu$ denotes the $\nu$-th Besselfunction of first kind and $\alpha_n>0$ is a normalization constant. The $f_n$ have exponential decay in the following way: For any $n\geq 1$ there exists $K=K_n>0$, $\mu=\mu_n$ such that
\begin{align*}
 |f_n(\eta)| \leq K \exp\left( -\mu \eta^{\frac{4}{3}}\right) \qquad  \text{ for all }\eta\geq 0,
\end{align*}
see \cite{Galaktionov, KochLamm}. Also estimates on the decay of the derivatives of $f_n$, and thus on the derivatives of $b$, are available \cite{KochLamm}. We will need to consider the backwards biharmonic heat kernel solving $\partial_t B = \Delta^2 B$ with concentration at $(x_0,t_0)$, which is given by 
\begin{align} \label{defn:B}
 B(x,t) := B_{(x_0,t_0)} (x,t) := b(x-x_0, t_0-t) = \alpha_n \frac{f_n\left(\frac{|x-x_0|}{(t_0-t)^{1/4}}\right)}{(t_0 - t)^{n/4}}.
\end{align}

\begin{rmk}
 The classical second order heat kernel on $\AR^n$ concentrating at $(x_0, t_0)=(0,0)$ is
 \[
  k(x,t) = \frac{1}{(4\pi t)^{n/2}} \exp\left(\frac{-|x|^2}{4t}\right)  \ \ \ \text{ for } (x,t)\in\AR^n\times (0,\infty)
 \]
and satisfies $\nabla k = -\frac{x}{2t} k$. It turns out that the biharmonic heat kernel satisfies a similar equation which we present now. This property will be crucial for our monotonicity results later on.
\end{rmk}

\begin{lemma} \label{lem:propbiheat}
The biharmonic heat kernel $b(x,t) = \frac{\alpha_n}{t^{\frac{n}{4}}} f_n\left(\frac{|x|}{t^{\frac{1}{4}}}\right)$  on $\AR^n$ satisfies 
 \begin{align} \label{eq:0}
\nabla \Delta b(x,t) = \frac{x}{4 t}b(x,t) \ \ \ \text{ on } \AR^n \times (0,\infty).
\end{align}
\end{lemma}
\begin{proof}
 We use the notation $f=f_n$ and
 \[
  \Delta^\eta f (\eta) := \frac{n-1}{\eta} f'(\eta) + f''(\eta)
 \]
which was introduced in \cite{AGG}. With $\eta=\frac{|x|}{t^{\frac{1}{4}}}$ we get that
 \begin{align*}
  \Delta^x (f(\eta)) &= \sum_i \partial^2_i (f(\eta)) = \sum_i \partial_i \left(f'(\eta) \frac{x_i}{|x|t^{\frac{1}{4}}}\right)\\
  &=  f''(\eta)\sum_i \frac{x_i^2}{|x|^2t^{\frac{1}{2}}} + f'(\eta) \frac{n}{|x|t^{\frac{1}{4}}} + \sum_i f'(\eta) \frac{x_i}{t^{\frac{1}{4}}}\left(-\frac{x_i}{|x|^3}\right)\\
  &= f''(\eta) \frac{1}{t^{\frac{1}{2}}} + f'(\eta) \frac{n-1}{t^{\frac{1}{2}}\frac{|x|}{t^{\frac{1}{4}}}}\\
  &= \frac{1}{t^{\frac{1}{2}}} \Delta^\eta f(\eta).
 \end{align*}
Using $\left(\Delta^\eta f_n(\eta)\right)' = \frac{\eta}{4} f_n(\eta)$ from \cite[Proof of Theorem~2.2]{AGG} we obtain
\begin{align*}
 \nabla^x\Delta^x (f(\eta))  = \frac{1}{t^{\frac{1}{2}}} \nabla^x \left(\Delta^\eta f(\eta)\right)
 &= \frac{1}{t^{\frac{1}{2}}} \left(\Delta^\eta f\right)' (\eta) \nabla^x \eta\\
 &= \frac{1}{t^{\frac{1}{2}}} \frac{\eta}{4} f(\eta) \frac{x}{|x|t^{\frac{1}{4}}}
 = \frac{x}{4 t} f(\eta)
\end{align*}
and 
\begin{align*}
 \nabla^x\Delta^x b(x,t) = \alpha_n \nabla^x\Delta^x \left(t^{-\frac{n}{4}} f(\eta)\right)
 = \alpha_n t^{-\frac{n}{4}}  \frac{x}{4 t} f(\eta)
  = \frac{x}{4 t} b(x,t).
\end{align*}

\end{proof}

\begin{cor} \label{cor:diffHarnack}
The biharmonic heat kernel $b(x,t) = \frac{\alpha_n}{t^{\frac{n}{4}}} f_n\left(\frac{|x|}{t^{\frac{1}{4}}}\right)$  on $\AR^n$ satisfies the differentiable Matrix Harnack type equation
\begin{align}\label{equ:MatrixHarnack}
b\cdot \nabla_i\nabla_j \Delta b - \nabla_i\Delta b\,\nabla_j b - \frac{\delta_{ij} b^2}{4t} = 0 \qquad \text{ on }\AR^n\times (0,\infty)
\end{align}
and (in the traced form)
 \begin{align} \label{eq:6}
 (\partial_t b)\cdot b +\frac{\Delta |\nabla b|^2}{2} -   |\nabla^2 b|^2 + \frac{n}{4t}  b^2 = 0 \qquad \text{ on }\AR^n\times (0,\infty).
\end{align}

\end{cor}

\begin{proof}
We have $\nabla_j\Delta b = \frac{x_j}{4t} b$ due to (\ref{eq:0}) and thus
\begin{align*}
b\cdot \nabla_i\nabla_j \Delta b  & = \frac{\delta_{ij}}{4t}b^2 + \frac{b\, x_i \nabla_j b}{4 t}  \\
 &=  \frac{\delta_{ij}}{4t}b^2 +  \nabla_i\Delta b\nabla_j b,
\end{align*}
which yields (\ref{equ:MatrixHarnack}). Summing up and tracing (\ref{equ:MatrixHarnack}), using $\partial_t b = -\Delta^2 b$ and computing
\begin{align*}
 \nabla\Delta b \cdot \nabla b = \frac{1}{2}  \Delta |\nabla b|^2 - |\nabla^2 b|^2
\end{align*}
brings us to (\ref{eq:6}).
\end{proof}

In the following, we prove further properties of the biharmonic heat kernel which will be used later.

\begin{lemma} \label{lem:bound1}
    The first zeros $\xi_{1}^{n}, \xi_{1}^{n+2}, ...$ of the generating functions $f_n, f_{n+2}, ...$ of the biharmonic heat kernel are in ascending order, i.e.,  $\xi_{1}^{n} < \xi_{1}^{n+2} < \xi_{1}^{n+4} < ...$.
\end{lemma}
\begin{proof}
    Let $\mu_{1}^{n}=0$ be the first maximum and $m_{1}^{n}$ the first minimum (here $f_n(m_{1}^{n})<0$) of $f_n$.
    First, observe that $f_{n}^{\prime}(\eta) <0$ for all $\eta\in (0,m_{1}^{n})$ \cite{AGG}. Thus, with $f_{n}^{\prime}(\eta) = -\eta\, f_{n+2}(\eta)$ \cite[Equation~(11)]{AGG}, we get that $f_{n+2}>0$ for \\ $\eta\in (0,m_{1}^{n})$. Since we have a positive maximum at $\eta=0$, we get the inequality for all $\eta\in [0,m_{1}^{n})$. Hence, $f_{n+2}(\eta)\neq 0$ for all $\eta\in [0,m_{1}^{n})$. In particular, the first zero of $f_{n+2}$ appears strictly after the first zero of $f_n$. Applying the recursive formula $f_{n}^{\prime}(\eta) = -\eta\, f_{n+2}(\eta)$ to all $n$ implies the result.
\end{proof}

We immediately get the following statement.

\begin{cor}
Let $\xi_1^n$ be the first zero of $f_n$. Then,
    for any $0<\xi(n)<\xi_1^n$, the terms $\frac{f_{n+2}}{f_n}$ and $\frac{f_{n+4}}{f_n}$ are bounded from below and above with positive bounds for all $\eta\in [0,\xi(n)]$.
\end{cor}

Later in our work, we need to control $\Delta^x B(x,t)$ which is directly related to $\Delta^\eta f_n (\eta)$. Thus, the following lemma will be useful.

\begin{lemma} \label{lem:bound2}
    For all $n\in\mathbb{N}$, there exist $\Tilde{\eta}_{1}(n)>0$ and $\Tilde{\eta}_{2}(n)>0$ as well as constants $c_1(n)>0$ and $c_2(n)>0$, all depending on $n$, such that 
    \begin{enumerate}
        \item $f_{n}(\eta) \geq c_1 > 0\, \quad \text{for all }\, \eta \leq \Tilde{\eta}_{1}\, \quad \text{and} $  \label{h-1} \\
        \item $\Delta^{\eta}f_{n}(\eta) = f_{n}^{\prime\prime}(\eta) + \frac{n-1}{\eta}f_{n}^{\prime}(\eta) \leq -c_2 < 0\, \quad \text{for all }\, \eta \leq \Tilde{\eta}_{2}\, . $ \label{h-2}
    \end{enumerate}
\end{lemma}

\begin{proof}
    We use the integral representation of $f_n$ (see \ref{equ:integral}), namely
    $$f_n(\eta) = \eta^{1-n}\int_{0}^{\infty} \exp\left(-s^4\right)(\eta s)^{\frac{n}{2}} J_{\frac{n-2}{2}}(\eta s)\,ds\, ,$$
    where $J_{\alpha} (x) = \big(\frac{x}{2}\big)^{\alpha} \frac{P_{\alpha}(x)}{\Gamma(\alpha +1)}$, with $P_{\alpha}(x):=\sum_{m=0}^{\infty}(-1)^m \frac{(\frac{x}{2})^{2m}\Gamma(\alpha +1)}{m! \Gamma(\alpha+m+1)}$ (see \cite[\href{https://dlmf.nist.gov/10.2}{(10.2(ii))}]{NIST:DLMF}). We furthermore use some rough estimates appearing in a recent result by Reif \cite{Reif}, namely 
    $$1-\frac{(\frac{x}{2})^2}{\alpha +1} =: P_{\alpha}^{1}(x)\leq P_{\alpha}(x)\leq 1\quad \text{for all } \alpha \geq -\frac{1}{2}\, \text{and } x\geq 0\, .$$
    Thus,
    \begin{align*}
        f_n(\eta) &\geq \frac{\eta^{1-n}}{\Gamma(\frac{n}{2})}\int_{0}^{\infty} \exp\left(-s^4\right)(\eta s)^{\frac{n}{2}} \big(\frac{\eta s}{2}\big)^{\frac{n-2}{2}} \Bigg(1-\frac{(\eta s)^2}{4( \frac{n-2}{2}+1)}\Bigg)\,ds \\
        & = \frac{2^{-\frac{n-2}{2}}}{\Gamma(\frac{n}{2})}\int_{0}^{\infty} \exp\left(-s^4\right) s^{n-1}\Bigg(1-\frac{(\eta s)^2}{2n}\Bigg)\,ds \\  
        & = \frac{2^{-\frac{n-2}{2}}}{\Gamma(\frac{n}{2})}\Bigg[\int_{0}^{\infty} \exp\left(-s^4\right) s^{n-1}\,ds - \frac{\eta^2}{2n}\int_{0}^{\infty}\exp\left(-s^4\right) s^{n+1}\,ds \Bigg] \\  
        & = \frac{2^{-\frac{n-2}{2}}}{\Gamma(\frac{n}{2})}\Bigg[\frac{1}{4}\Gamma\bigg(\frac{n}{4}\bigg) - \frac{\eta^2}{8n}\Gamma\bigg(\frac{n+2}{4}\bigg) \Bigg]\, ,
    \end{align*}
    where we have substituted $t=s^4$ and used $\Gamma(x) = \int_{0}^{\infty}e^{-t}t^{x-1}\,dt$. Now, 
    \begin{align*}
\eta \,<\,  2\,\sqrt{2\,\frac{\Gamma(\frac{n+4}{4})}{\Gamma(\frac{n+2}{4})}}\, &\Longleftrightarrow  \eta^2 \,<\,  \frac{8n}{4} \,\frac{\Gamma\Big(\frac{n}{4}\Big)}{\Gamma\Big(\frac{n+2}{4}\Big)} \\
& \Longleftrightarrow \frac{\eta^2}{8n} \,\Gamma\Big(\frac{n+2}{4}\Big) \,<\,\frac{1}{4}\,\Gamma\Big(\frac{n}{4}\Big) \\
&\Longrightarrow f_{n}(\eta) \,>\, 0,
    \end{align*}
    where we have used $x\Gamma(x)=\Gamma(x+1)$. Also note that $\Gamma(x)$ is positive for all non-negative $x$ and therefore $\eta$ will be positive.   \\
For (\ref{h-2}), we know that $f_{n}^{\prime}(\eta)=-\eta f_{n+2}(\eta)$. \\ Thus,
$\Delta^{\eta}f_{n}(\eta) = \eta^2 f_{n+4}(\eta) - n f_{n+2}(\eta)\, $
and we get\enlargethispage{2em}
\begin{align*}
    \Delta^{\eta}f_{n}(\eta)  &= \eta^2 \eta^{1-(n+4)}\int_{0}^{\infty} \exp\left(-s^4\right)(\eta s)^{\frac{n+4}{2}} J_{\frac{n+2}{2}}(\eta s)\,ds \\
			      &\quad - n \eta^{1-n}\int_{0}^{\infty} \exp\left(-s^4\right)(\eta s)^{\frac{n}{2}} J_{\frac{n}{2}}(\eta s)\,ds  \\
    &= \eta^{1-\frac{n}{2}}\int_{0}^{\infty} \exp\left(-s^4\right) s^{\frac{n+4}{2}} \big(\frac{\eta s}{2}\big)^{\frac{n+2}{2}} \frac{P_{\frac{n+2}{2}}(\eta s)}{\Gamma(\frac{n+4}{2})}\,ds\\
    &\quad -n\eta^{-\frac{n}{2}}\int_{0}^{\infty} \exp\left(-s^4\right) s^{\frac{n+2}{2}} \Big(\frac{\eta s}{2}\Big)^{\frac{n}{2}} \frac{P_{\frac{n}{2}}(\eta s)}{\Gamma(\frac{n+2}{2})}\,ds   \\
    &= \frac{1}{2^\frac{n}{2}} \Bigg(\frac{\eta^2}{2}\int_{0}^{\infty} \exp\left(-s^4\right) s^{n+3} \frac{P_{\frac{n+2}{2}}(\eta s)}{\Gamma(\frac{n+4}{2})}\,ds\\
    &\quad\qquad  -n\int_{0}^{\infty} \exp\left(-s^4\right) s^{n+1} \frac{P_{\frac{n}{2}}(\eta s)}{\Gamma(\frac{n+2}{2})}\,ds \Bigg)\, .   \\
\end{align*}
Using the same estimates as above, we obtain
\begin{align*}
    2^{\frac{n}{2}}\Delta^{\eta}f_{n}(\eta) &\leq  \frac{\eta^2}{2\Gamma(\frac{n+4}{2})}\int_{0}^{\infty} \exp\left(-s^4\right) s^{n+3} \,ds \\
					    &\quad -\frac{n}{\Gamma(\frac{n+2}{2})}\int_{0}^{\infty} e^{-s^4} s^{n+1} \Big(1-\frac{(\eta s)^2}{2n +4}\Big)\,ds  \\
    &=  \Bigg(\frac{\eta^2}{2\Gamma(\frac{n+4}{2})}+\frac{n \eta^2}{2(n+2)\Gamma(\frac{n+2}{2})}\Bigg)\int_{0}^{\infty} \exp\left(-s^4\right) s^{n+3} \,ds\\
    &\quad -\frac{n}{\Gamma(\frac{n+2}{2})}\int_{0}^{\infty} \exp\left(-s^4\right) s^{n+1}\,ds    \\
    &= \eta^2 \Bigg(\frac{n+2}{4\Gamma(\frac{n+4}{2})}\Bigg)\int_{0}^{\infty} \exp\left(-s^4\right) s^{n+3} \,ds\\
    &\quad -\frac{n}{\Gamma(\frac{n+2}{2})}\int_{0}^{\infty} \exp\left(-s^4\right) s^{n+1}\,ds    \\ \pagebreak[0]
    &= \eta^2 \Bigg(\frac{n+2}{4\Gamma(\frac{n+4}{2})}\Bigg)\int_{0}^{\infty}\frac{1}{4} \exp\left(-t\right) t^{\frac{n}{4}} \,dt \\
    &\quad -\frac{n}{\Gamma(\frac{n+2}{2})}\int_{0}^{\infty} \frac{1}{4} \exp\left(-t\right) t^{\frac{n-2}{4}}\,dt   \\ \pagebreak[0]
    &= \eta^2 \frac{(n+2)}{16}\,\frac{\Gamma(\frac{n+4}{4})}{\Gamma(\frac{n+4}{2})} - \frac{n}{4}\,\frac{\Gamma(\frac{n+2}{4})}{\Gamma(\frac{n+2}{2})}\, .
\end{align*}
Now, 
\begin{align*}
    &\quad& \Delta^{\eta}f_{n}(\eta) \,<&\, 0 \\
    &\Leftarrow& \eta^2 \frac{(n+2)}{16}\,\frac{\Gamma(\frac{n+4}{4})}{\Gamma(\frac{n+4}{2})} \,<&\,  \frac{n}{4}\,\frac{\Gamma(\frac{n+2}{4})}{\Gamma(\frac{n+2}{2})}    \\
    &\Leftrightarrow& \hfill \eta^2 \,<&\,  \frac{4n}{(n+2)}\frac{\Gamma(\frac{n+2}{4})}{\Gamma(\frac{n+4}{4})} \frac{\Gamma(\frac{n+4}{2})}{\Gamma(\frac{n+2}{2})}  \\
    &\Leftrightarrow& \hfill \eta^2 \,<&\,  \frac{4^2 n}{(n+2)^2}\frac{\frac{(n+2)}{4}\Gamma(\frac{n+2}{4})}{\Gamma(\frac{n+4}{4})} \frac{\Gamma(\frac{n+4}{2})}{\Gamma(\frac{n+2}{2})}  \\    
    &\Leftrightarrow& \eta \,<&\,  \frac{4}{(n+2)}\sqrt{\frac{n\, \Gamma(\frac{n+6}{4})\Gamma(\frac{n+4}{2})}{\Gamma(\frac{n+4}{4})\Gamma(\frac{n+2}{2})}}\, .    \\
\end{align*}
\end{proof}

\begin{rmk}
    In Lemma~\ref{lem:bound2}, $\Tilde{\eta}_{1}(n)$ is given $\Tilde{\eta}_{1}(n)=2\sqrt{2\frac{\Gamma(\frac{n+4}{4})}{\Gamma(\frac{n+2}{4})}}$ and can furthermore be relaxed by $\Tilde{\eta}_1(n)<2\sqrt{2}<2\sqrt{2\frac{\Gamma(\frac{n+4}{4})}{\Gamma(\frac{n+2}{4})}}$ for all $n\geq3$. This follows from the fact that $\Gamma(x)$ is monotone increasing for $x\geq \frac{3}{2}$, which for $\Gamma\big(\frac{n+2}{4}\big)$ is the case when $n\geq4$. However, $\frac{3}{4}\Gamma\big(\frac{3}{4}\big)=\Gamma\big(\frac{7}{4}\big) > \Gamma\big(\frac{5}{4}\big)=\frac{1}{4}\Gamma\big(\frac{1}{4}\big)$, too \cite[\href{https://dlmf.nist.gov/5.4.E1}{(5.4.(i))}]{NIST:DLMF}. So, $\Gamma\big(\frac{n+4}{4}\big) > \Gamma\big(\frac{n+2}{4}\big)$ for all $n\geq 3$. The values for $n=1$ and $n=2$ can be calculated as $\Tilde{\eta}_1(1)=2\sqrt{2\frac{\Gamma(\frac{5}{4})}{\Gamma(\frac{3}{4})}}-\frac{1}{1000}\approx 2.43156...$ and $\Tilde{\eta}_1(2)=2\sqrt{2\frac{\Gamma(\frac{6}{4})}{\Gamma(\frac{4}{4})}}-\frac{1}{1000}\approx 2.66167...$. Then, for $n\geq 3$, $c_1(n)$ is given by $c_1(n) = \frac{2^{-\frac{n-2}{2}}}{n\, \Gamma(\frac{n}{2})}\Big(\Gamma\big(\frac{n+4}{4}\big)-\Gamma\big(\frac{n+2}{4}\big)\Big)$. The values for $n=1$ and $n=2$ are $c_1(1)\approx 0.000594736...$ and $c_1(2)\approx 0.000333013...$ . Furthermore, $\Tilde{\eta}_{2}(n)$ is given by $\Tilde{\eta}_{2}(n)=\frac{4}{(n+2)}\sqrt{n\frac{\Gamma(\frac{n+6}{4})}{\Gamma(\frac{n+4}{4})}\frac{\Gamma(\frac{n+4}{2})}{\Gamma(\frac{n+2}{2})}}$, and by the same reasoning as above it can be relaxed by $\Tilde{\eta}_{2}(n)<\frac{4\sqrt{n}}{(n+2)}<\frac{4}{(n+2)}\sqrt{n\frac{\Gamma(\frac{n+6}{4})}{\Gamma(\frac{n+4}{4})}\frac{\Gamma(\frac{n+4}{2})}{\Gamma(\frac{n+2}{2})}}$ for all $n\in\mathbb{N}$. Then, $c_2(n)$ is given by $c_2(n)=-2^{-\frac{n}{2}}\frac{n}{n+2}\Big(\frac{\Gamma(\frac{n+4}{4})}{\Gamma(\frac{n+4}{2})} - \frac{\Gamma(\frac{n+6}{4})}{\Gamma(\frac{n+2}{2})}\Big)$. However, due to some very rough estimates, these bounds are by no means sharp. Nevertheless, this is good enough for our purposes.
\end{rmk}

\section{A pre-entropy and its monotonicity behavior}\label{sect:monotonicity}

\begin{rmk} \label{rmk:solitons}
 The equation $\partial_t u = -\lap^2 u - f(u)$ is invariant under shifts $x\mapsto x-x_0$, $t\mapsto t-t_0$ and also under parabolic rescaling in the following sense. We define $u_R(x,t) := u(Rx,R^4t) $ for $R>0$, then $u_R$ is a solution of the biharmonic map heat flow whenever $u$ was a solution. \emph{Solitons} or \emph{self-similar solutions} are solutions of the biharmonic map heat flow that are invariant under this rescaling, that is, $\frac{\partial u_R}{\partial R}  (x,t)\big|_{R=1} =0$ for all $(x,t)\in \AR^n\times I$ for an interval $I\subset \AR$. We compute the equation that $u$ needs to satisfy to be a soliton:
 \begin{align*}
  0= \frac{\partial u_R}{\partial R}  (x,t)\big|_{R=1} & = \left( (\nabla u) (Rx,R^4t) \cdot x + 4 R^3\, t\,(\partial_t u)(Rx,R^4t)\right)\big|_{R=1}\\
  &= \nabla u(x,t)\cdot x + 4 t \partial_t u (x,t).
 \end{align*}
So, a soliton is a solution of the biharmonic map heat flow satisfying
\[
 \Sol(u) := \nabla u\cdot (x-x_0) + 4 (t-t_0) \partial_t u =0 \ \ \ \text{if either } t< t_0 \text{ or } t> t_0.
\]
These include self-shrinkers (for example in the form \\ $u(x,t) = v\left(\frac{|x|}{(-t)^{1/4}}\right), t\in (-\infty, 0)$) and expanders (for example in the form $v\left(\frac{|x|}{t^{1/4}}\right), t\in (0, \infty)$). For results about solitons for the harmonic map heat flow and the Yang-Mills heat flow, see for example the work of Gastel \cite{GastelSolitons}, Chen-Zhang \cite{SolitonsYM}, Weinkove \cite{Weinkove-YMSol}, Glogi\'c-Kistner-Sch\"orkhuber \cite{Self-similar}, Germain-Rupflin \cite{GermainExp}, Deruelle-Lamm \cite{DeruelleLammExp}, Deruelle \cite{DeruelleExp} and Biernat-Bizo\'{n} \cite{BiernatExp}. In the closed setting we know that there are no singularities for the extrinsic biharmonic map heat flow for $n\leq 3$ from the work of Lamm \cite{LammExtr}, so no self-shrinkers exist in these dimensions in the closed setting. To the authors' knowledge, there are no results about the (non-)existence and properties of expanders for the biharmonic map heat flow.
Interestingly, a result of Galaktionov \cite{Galaktionov2} shows that there are self-shrinkers (and thus finite time singularities) of a fourth order semilinear equation on space dimension $1$. 
\end{rmk}

\begin{defn} \label{defn:pre-entropy}
Let $u:\AR^n\times [0,T) \to N \hookrightarrow \AR^k$ be a fourth order $L^2$-gradient flow of a (geometric) energy $E$. For $R>0$ and $t_0\in (0,T)$ we define a slice $S_R$ and a horizontal layer $T_R$ by
\begin{align*}
S_R(t_0)&:= \AR^n \times\{t = t_0-R^4\}, \\
 T_R(t_0)&:=\{(x,t)\in \AR^n \times \AR : t_0 - 16R^4 < t < t_0-R^4 \}
\end{align*}
 and let $\varphi \in C^\infty(\AR^n)$ be a smooth, non-negative, bounded function (for example a smooth cut-off function). A functional $\Psi (u,R,\varphi)$ is called a \emph{pre-entropy} if the following three properties are satisfied:
 \begin{enumerate}
 \item  $\Psi (u,R,\varphi)$ is a non-negative integral over $S_R(t_0)$ or $T_R(t_0)$, involving $u$ and its derivatives such that $\Psi $ is scaling invariant with respect to the rescaling $
x := Ry, t := R^4s$. That is, with $u_R(x,t): = u(Rx,R^4t)$ and $\varphi_R(x,t): = \varphi(Rx,R^4t)$ it satisfies
\[
\Psi(u, R, \varphi)= \Psi(u_R, 1,\varphi_R).
\]
  \item shrinking solitons (see Remark~\ref{rmk:solitons}) are critical points of~$\Psi$,~i.e.
\[\frac{d}{dR} \Psi(\tilde u, R,1) =0
\]
for shrinking solitons $\tilde u$, and
\item there are $\alpha, \beta >0$ and $c_1, c_2\geq 0$ such that a (pseudo-)monotonicity formula of the form
\[
 \Psi( u, R_1, \varphi)\leq \exp \left(c_1 (R_2^\alpha - R_1^\alpha)\right)\Psi(u,R_2,\varphi) + c_2 \left(R_2^\beta-R_1^\beta\right) E_0
\]
holds for $R_1<R_2$. Here, $E_0= E(u(\cdot,0)) + Q$ is the initial energy plus possibly bounds on other energies ideally only depending on the initial time.
 \end{enumerate}

\end{defn}

 \begin{rmk}
 The monotone object in Huisken's monotonicity formula for the Mean Curvature flow \cite{HuiskenMono} is often called \emph{Gau\ss ian area} \cite[Chapter~10]{AndrewsBook} or $\mathcal{F}$-functional (see for example the work of Colding-Minicozzi \cite{ColdingMinicozzi}). As stated by Huisken in \cite{HuiskenMono}, this monotonicity formula was inspired by work of Giga and Kohn on semilinear heat equations \cite{GigaKohn1, GigaKohn2}. Another common name for a pre-entropy in the second order case is \emph{relative entropy} or \emph{weighted entropy}. Following Magni-Mantegazza \cite{Magni}, Colding and Minicozzi defined the \emph{entropy} for solutions of the Mean Curvature flow as a certain supremum over $(x_0, t_0)$ of the Gau\ss ian area \cite{ColdingMinicozzi}. Since we do not weight our energies with a Gau\ss ian density but with the biharmonic heat kernel and as we do not take a supremum over $(x_0, t_0)$, we prefer the more general notion \emph{pre-entropy}.
 \end{rmk}

\begin{defn}
We let $u:\AR^n \times [0,T) \to N\hookrightarrow \AR^k$  be a smooth solution of the extrinsic biharmonic map heat flow $\AR^n$, $\partial_t u = -\lap^2 u - f(u)$, $u(\cdot, 0)=u_0$, see (\ref{equ:flow}) for the definition of $f(u)$. 
For $(x_0,t_0) \in\AR^n\times (0,T]$, consider the backwards biharmonic heat kernel $B$ as defined in (\ref{defn:B}). We further let $\vp:\AR^n\times[0,T)\to\AR$ to be a non-negative, smooth, bounded function with compact support. 
We define our functionals $\Psi$ and $\Phi$ which will turn out to be \emph{pre-entropies} in certain settings as follows:
\begin{align}
\label{eq:entropydef}
\Psi \prs{u, R ,\vp} &:= \tfrac{1}{2} \int_{T_R(t_0)} \brs{\lap u(x,t)}^2  B_{(x_0,t_0)}(x,t) \vp(x,t) \,dx\,dt \\
\notag &\quad\ - \tfrac{1}{2} \int_{T_R(t_0)}\brs{\nabla u(x,t)}^2 \lap B_{(x_0,t_0)}(x,t) \vp(x,t) \,dx\,dt.
\end{align}
and
\begin{align}
\label{eq:entropydef2}
\Phi \prs{u, R ,\vp} &:= \tfrac{R^4}{2} \int_{S_R(t_0)} \brs{\lap u(x,t)}^2  B_{(x_0,t_0)}(x,t) \vp(x,t) \,dx\\
\notag &\quad\ - \tfrac{R^4}{2} \int_{S_R(t_0)}\brs{\nabla u(x,t)}^2 \lap B_{(x_0,t_0)}(x,t) \vp(x,t) \,dx.
\end{align}
\end{defn}
For the computations it suffices to consider the case $(x_0,t_0)=(0,0)$ due to the invariance under shifting.
Let us first check some basic properties of $\Psi$ and $\Phi$.

\begin{lemma}\label{lem:scaling}
 The functionals defined in (\ref{eq:entropydef}) and (\ref{eq:entropydef2}) are invariant under the  scaling  $x := Ry,$ $t := R^4s$.
 That is, with $u_R(y,s):= u(Ry, R^4s)$ and $\varphi_R(x,t):= \varphi(Rx, R^4t)$ they satisfy
 \[
  \Psi(u,R,\varphi) = \Psi(u_R,1,\varphi_R), \ \ \ \Phi(u,R,\varphi) = \Phi(u_R,1,\varphi_R)
 \]

\end{lemma}

\begin{proof}
 The scaling law leads to the following transformations
 \begin{align*}
dx &= R^n \,dy \\
dt &= R^4 \,ds \\
B(x,t)\big|_{(Ry, R^4s)}& = \frac{1}{R^n} B(y,s)\\
B(x,t) \,dx\,dt &= R^4 B(y,s) \,dy\,ds \\
\lap u (Ry, R^4s)& = \frac{1}{R^2} \lap u_R (y,s).
\end{align*}
For the Laplacian of $B$, we proceed as in Lemma~\ref{lem:propbiheat} and use that $\eta = \frac{|x|}{(-t)^{1/4}}$ and $\Delta^\eta f_n(\eta)$ is invariant under our rescaling procedure. We get
\[
 \Delta^x B(x,t) = \frac{\alpha_n}{(-t)^{(n+2)/4}} \Delta^\eta f_n(\eta)
\]
and thus
\[
 \Delta^x B(x,t)\big|_{(Ry,R^4s)} = \frac{\alpha_n}{(-R^4 s)^{(n+2)/4}} \Delta^\eta f_n(\eta) = \frac{1}{R^{2+n}} \Delta^y B(y,s).
\]
The scaling invariance now follows from transformation of the integrals.
\end{proof}

\subsection{The linear case} In the case $N=\AR^k$, our evolution equation reduces to the scalar bi-heat equation on $\AR^n$. We consider this case separately to clarify which properties of the monotonicity formula arise from the underlying structure and which come from the nonlinearity. First we derive some classical a priori estimates which we need later. 

\begin{lemma} \label{lem:flowprop}
Let $u:\AR^n \times [0,\infty) \to\AR$ be a smooth solution of the biheat equation on $\AR^n$ with smooth initial datum $u_0:\AR^n \to \AR$.
    Assume
    $$\int_{\mathbb{R}^n} |\nabla^{l} u (x,t)|^2 \,dx < \infty\, \quad \text{for all }\, t\in[0,\infty)\, \text{and for all }\, l\in\lbrace 1,2,3\rbrace\, . $$
    Then, for each $t\in [0,\infty)$ we get the a priori estimates
    \begin{align}
        \tfrac{1}{2} \int_{\mathbb{R}^n} |\Delta u|^2 (x,t)\,dx 
       + \int_0^t\int_{\AR^n} |\partial_t u|^2 (x,s) \,dx\,ds 
       &\leq \tfrac{1}{2} \int_{\mathbb{R}^n} |\Delta u_0|^2 \,dx\, , \label{est:lap-2} \\
         \tfrac{1}{2} \int_{\mathbb{R}^n} |\nabla u|^2 (x,t)\,dx + \int_0^t\int_{\AR^n} |\nabla\Delta u|^2(x,s) \,dx\,ds &\leq \tfrac{1}{2}  \int_{\mathbb{R}^n} |\nabla u_0|^2 \,dx\, , \label{est:N-2}      
    \end{align}
    If we further assume $\int_{\mathbb{R}^n} |\nabla^4 u (x,t)|^2 \,dx < \infty$ for each $t$, then we have the inequality
 \begin{align}
  \int_{\mathbb{R}^n} |\nabla\Delta u|^2 (x,t)\,dx &\leq \int_{\mathbb{R}^n} |\nabla\Delta u_0|^2 \,dx\, . \label{est:Nlap-2}
 \end{align}
\end{lemma}

\begin{proof}
   We omit the proof of (\ref{est:lap-2}) because it works analogously to that of~(\ref{est:N-2}). We start with (\ref{est:N-2}).
    We use classical cut-off functions \\
    $\psi_r(x)=\psi_0\big(\frac{|x|}{r}\big)$ with
    $$\psi_0(\alpha)= \begin{cases}
        0 & \text{for }\, \alpha \geq 1   \\
        1 & \text{for }\, \alpha \leq \frac{1}{2}
    \end{cases}\,$$
    and $\psi_0$ being smooth and monotone for $\alpha\in[\frac{1}{2},1]$. 
    Now,
 \begin{align*}
       & \frac{\text{d}}{\text{d}t} \frac{1}{2}\int_{\mathbb{R}^n} \psi_r^4 |\nabla u|^2 \,dx   \\
       &= \int_{\mathbb{R}^n} \psi_r^4 \langle\nabla \partial_t u, \nabla u\rangle \,dx     \\
        &= -\int_{\mathbb{R}^n} \psi_r^4 \langle\nabla\Delta^2 u, \nabla u\rangle \,dx     \\
       &= - \int_{\mathbb{R}^n} \psi_r^4 |\nabla\Delta u|^2 \,dx - 4 \int_{\mathbb{R}^n} \psi_r^3\Delta\psi_r \langle\nabla\Delta u, \nabla u\rangle \,dx \\
       &\quad -12 \int_{\AR^n} \psi_r^2 |\nabla\psi_r|^2 \langle \nabla\Delta u,\nabla u\rangle  \,dx \\
       &\quad - 8 \int_{\AR^n} \psi_r^3  \nabla_i \psi_r \langle \nabla \nabla_i u , \nabla\Delta u \rangle \,dx + \textit{vanishing bdry terms}   \\
       &\leq - (1-3\varepsilon) \int_{\mathbb{R}^n} \psi_r^4 |\nabla\Delta u|^2 \,dx      \\
 &\quad + c(\varepsilon) \left(\int_{\mathbb{R}^n} |\Delta\psi_r|^2 |\nabla u|^2 \,dx  + \int_{\AR^n} |\nabla \psi_r|^4 |\nabla u|^2 \,dx 
 \nonumber + \int_{\AR^n} |\nabla \psi_r|^2 |\nabla^2 u|^2 \,dx \right)
    \end{align*}
    where we have used partial integration with vanishing boundary terms because of the compact support of the cutoff function $\psi_r$, the equation $\partial_t u = -\Delta^2 u$, $|\psi_r|\leq 1$ as well as Young's inequality.
    Integration from $0$ to $t$ gives
    \begin{align*}
    \frac{1}{2}\int_{\mathbb{R}^n} &\psi_r^4 |\nabla u|^2 (x,t) \,dx   + (1-3\varepsilon) \int_0^t \int_{\AR^n} \psi_r^4|\nabla \Delta u|^2 (x,s) \,dx \,ds\\
    & \leq \frac{1}{2}\int_{\mathbb{R}^n} \psi_r^4 |\nabla u|^2 (x,0) \,dx   + c(\varepsilon) \int_{0}^{t}\int_{\mathbb{R}^n} |\Delta\psi_r|^2 |\nabla u|^2 (x,s)\,dx\,ds \\
    &\quad +  c(\varepsilon) \int_{0}^{t}\int_{\mathbb{R}^n} |\nabla\psi_r|^4 |\nabla u|^2 (x,s)\,dx\,ds \\
    &\quad + c(\varepsilon)  \int_{0}^{t}\int_{\mathbb{R}^n} |\nabla\psi_r|^2 |\nabla^2 u|^2 (x,s)\,dx\,ds.
    \end{align*}
Next, we take $r\to \infty$ and see that
\begin{align*}
&\frac{1}{2}\int_{\mathbb{R}^n} |\nabla u|^2 (x,t) \,dx   + (1-3\varepsilon) \int_0^t \int_{\AR^n}|\nabla \Delta u|^2 (x,s) \,dx \,ds \\
&\qquad \leq \frac{1}{2}\int_{\mathbb{R}^n} |\nabla u|^2 (x,0) \,dx
\end{align*}
because all terms involving a derivative of $\psi_r$ will converge to zero. In a last step, we let $\varepsilon\to 0$ and get (\ref{est:N-2}).

For (\ref{est:Nlap-2}), we repeat the first steps from before, now for $\nabla\Delta u$ instead of $\nabla u$. We arrive at
    \begin{align*}
        \frac{\text{d}}{\text{d}t} &\frac{1}{2}\int_{\mathbb{R}^n} \psi_r^4 |\nabla\Delta u|^2 \,dx 
        \leq - (1-3\varepsilon) \int_{\mathbb{R}^n} \psi_r^4 |\nabla\Delta^2 u|^2 \,dx      \\
 &\quad + c(\varepsilon) \biggl(\int_{\mathbb{R}^n} |\Delta\psi_r|^2 |\nabla\Delta u|^2 \,dx  \\
 &\qquad \qquad + \int_{\AR^n} |\nabla \psi_r|^4 |\nabla \Delta u|^2 \,dx 
  + \int_{\AR^n} |\nabla \psi_r|^2 |\nabla^2 \Delta u|^2 \,dx \biggr).
    \end{align*}
    Now, we estimate the first term on the right hand side by zero and proceed with the rest as in the proof of (\ref{est:N-2}) with integration, $r\to\infty$ and use that $\int_{\AR^n} |\nabla^4 u|^2 \,dx$ is assumed to be bounded to get 
    $$\frac{1}{2}\int_{\mathbb{R}^n}|\nabla\Delta u|^2 (x,t) \,dx \leq \frac{1}{2}\int_{\mathbb{R}^n}|\nabla\Delta u|^2 (x,0) \,dx\, ,$$
    which is what we wanted to prove.
\end{proof}

 \begin{rmk}
 \begin{enumerate}
     \item  In the above proof for (\ref{est:Nlap-2}), we already needed \\ $\nabla^4 u \in L^2$. If we additionally assume  $\int_{\AR^n}|\nabla^5 u|^2 (x,t)\,dx <\infty$ for each $t$, then we also get 
 \begin{align*}
        \tfrac{1}{2}  \int_{\mathbb{R}^n} |\nabla\Delta u|^2 (x,t)\,dx + \int_0^t\int_{\AR^n} |\nabla\Delta^2 u|^2 (x,s)\,dx\,ds \leq \tfrac{1}{2} \int_{\mathbb{R}^n} |\nabla\Delta u_0|^2 \,dx\, 
\end{align*} 
by the proof above.
\item For a smooth biharmonic map heat flow on a closed manifold $M$, $L^2$-bounds on any $\nabla^l u(\cdot, t)$ come for free. Also boundary terms do not appear so that integration by parts is easier to use. But then, the non-linear terms appear. For the extrinsic biharmonic map heat flow, a version will be proven in Section~\ref{sect:nonlinear}. A version for the intrinsic biharmonic map heat flow with non-positive sectional curvature of the ambient space can be found in \cite{Lamm2005}.
 \end{enumerate}

 \end{rmk}

The variable $R$ represents the difference to the concentration time $t_0$, and $t_0$ will be arbitrary within the time interval of smooth existence. Therefore, we restrict ourselves in several results to the case where $R\leq 1$. As a consequence, only the lowest powers $R^p$ appear in front of the  $L^2$-energies.

 \begin{thm}\label{mainthm-derivatives-withoutSUP}
  Let $u:\AR^n \times [0,\infty) \to\AR$ be a smooth solution of the bi-heat equation on $\AR^n$ with initial datum $u_0:\AR^n \to \AR$ satisfying $ \int_{\AR^n} |\nabla^l u (x, t)|^2 dx<\infty$ for each $t\in [0,\infty)$ and $l=1,2,3$. Then, for any $(x_0,t_0)\in\AR^n \times (0,\infty)$ there is a smooth cut-off function \\ $\varphi \in C^\infty_c (\AR^n\times [0,\infty))$ depending only on $f_n$ such that both terms in the definition of $\Psi$ and $\Phi$ are non-negative, and there is a constant $C=C(f_n)$ such that, for every $R$ with $0 <R < t_0^{1/4}/2$ we have 
      \begin{align} 
\label{mainInequ-0-withoutSUP} 
\tfrac{\del}{\del R} \brk{\Psi \prs{u,R,\vp}}
& \geq \tfrac{1}{2}\tfrac{1}{ R} \int_{T_R} \tfrac{ \brs{\Sol u}^2}{4 \brs{t}} B \vp \,dx\,dt  - C  R^{1-n} \int_{\AR^n}|\N u_0|^2 dx\\
&\hsp - C  R^{3-n} \int_{\AR^n}|\lap u_0|^2 dx.
\end{align}
where $\Sol u: = \nabla u\cdot (x-x_0) + 4 (t-t_0)\partial_t u$.
In particular, for $n=1$, after integration from $R_1$ to $R_2$ with $0< R_1 <R_2 \leq t_0^{1/4}/2$ the integral $\Psi$ is a pre-entropy due to Definition~\ref{defn:pre-entropy} with monotonicity
\begin{align}  \label{integratedinequ}
   \Psi( u, R_1, \vp) &\leq \Psi (u, R_2,\vp)  + C ({R_2}- R_1)\, \int_{\AR^n}| u_0'|^2 \,dx \\
		  \notag     &\quad + C\left(R_2^3- R_1^3\right) \int_{\AR^n} | u_0''|^2 \,dx.
  \end{align}

 \end{thm}

\begin{proof} of Theorem~\ref{mainthm-derivatives-withoutSUP}.
W.l.o.g.\ we can assume $t_0=0$ and $x_0 =0$. For that, we shift our original time interval. Note that this implies that $u$ is now defined on $\AR^n\times [\hat T,\infty)$ with $\hat T< 0=t_0$. With a slight abuse of notation we keep the notation $u_0$ for the initial data meaning now $u_0:= u(\cdot, \hat T)$. We use Lemma~\ref{lem:scaling} and compute
\begin{align*} 
&\frac{\partial}{\partial R}  \Psi(u,R,\varphi)  = \frac{d}{dR} \Psi(u_R,1,\varphi_R)\\
& \ \ = \int_{T_1} \Delta u_R \cdot \Delta \left(\frac{d}{dR} u_R\right) B\, \varphi_R \, dy\, ds + \tfrac{1}{2} \int_{T_1} \brs{\Delta u_R}^2 B \left(\frac{d}{dR} \varphi_R\right) \, dy \,ds\\
& \ \ - \int_{T_1} \nabla u_R \left(\frac{d}{dR} u_R\right) \Delta B \varphi_R \, dy \,ds - \tfrac{1}{2}\int_{T_1} |\nabla u_R|^2  \Delta B\left(\frac{d}{dR} \varphi_R\right) \, dy \,ds.
\end{align*}
We will choose our cut-off function in the form $\varphi(x,t) = \varphi_0(\tfrac{|x|}{|t|^{1/4}})$ which implies $\varphi_R(x,t)=\varphi(Rx, R^4t) =\varphi(x,t)$ for all $x$ and $t$. So we get $\frac{d}{dR} \varphi_R =0$ and continue the computation above via integration by parts:

\begin{align*}
\frac{\partial}{\partial R} & \Psi(u,R,\varphi) = \frac{d}{dR} \Psi(u_R,1,\varphi_R)\\  \displaybreak[0]
& = \int_{T_1} \Delta^2 u_R \cdot  \left(\frac{d}{dR} u_R\right) B\,\varphi_R\, dy\, ds\\
&\quad + 2 \int_{T_1}\nabla\Delta u_R \cdot \nabla [B\varphi_R ]  \left(\frac{d}{dR} u_R\right)  \, dy\, ds\\ 
& \quad + \int_{T_1} \Delta u_R \left(\frac{d}{dR} u_R\right) \Delta [B\varphi_R] \, dy \,ds\\
&  \quad + \int_{T_1} \Delta u_R\cdot \left(\frac{d}{dR} u_R\right) \Delta B\, \varphi_R \, dy \,ds \\
&\quad + \int_{T_1} \nabla u_R \cdot \nabla [\Delta B\, \varphi_R] \left(\frac{d}{dR} u_R\right) \, dy\, ds\\ \displaybreak[0]
& = \int_{T_1} \Delta^2 u_R \cdot  \left(\frac{d}{dR} u_R\right) B\,\varphi_R\, dy\, ds \\
&\quad + 2 \int_{T_1}\nabla\Delta u_R \cdot \nabla B   \left(\frac{d}{dR} u_R\right) \varphi_R \, dy \,ds\\
& \quad +  2 \int_{T_1}\nabla\Delta u_R \cdot \nabla \varphi_R \left(\frac{d}{dR} u_R\right) B  \, dy\, ds \\
&\quad + \int_{T_1} \Delta u_R \left(\frac{d}{dR} u_R\right) \left[\Delta B\,\varphi + 2 \nabla B\cdot\nabla\varphi_R + B\Delta\varphi_R\right] \, dy\, ds\\
& \quad + \int_{T_1} \Delta u_R\cdot \left(\frac{d}{dR} u_R\right) \Delta B\, \varphi_R \, dy \,ds\\
&\quad + \int_{T_1} \nabla u_R \cdot \nabla \Delta B  \left(\frac{d}{dR} u_R\right) \varphi_R \, dy \,ds\\
& \quad +  \int_{T_1} \nabla u_R \cdot \nabla\varphi_R \,\Delta B \left(\frac{d}{dR} u_R\right) \, dy\, ds.
\end{align*}
 With the notation $\Sol \psi := \nabla \psi \cdot x + 4 t\partial_t\psi$ for any function $\psi:\AR^n \to \AR^L$, we see that 
\begin{align*}
\frac{d}{dR} u_R (y,s) &= y\cdot \nabla u(Ry, R^4s) + 4  R^3 s\, \partial_t u(Ry, R^4s)\\
& =\tfrac{1}{R}\left( (Ry)\cdot  \nabla u(Ry, R^4s) + 4  (R^4 s) \partial_t u(Ry, R^4s) \right)\\
& = \tfrac{1}{R} \Sol u\big|_{(Ry,R^4s)}.
\end{align*}
 We transform back onto $T_R$ via
\begin{align}
 \label{equ:halfway}
	\frac{\partial}{\partial R}  \Psi(u,R,\varphi) &= \tfrac{1}{R} \int_{T_R} \Delta^2 u \cdot   (\Sol u) B\,\varphi\, dx\, dt  \\
		\notag 				       &\quad +  \tfrac{2}{R}\int_{T_R}\nabla\Delta u \cdot \nabla B \,   (\Sol u) \varphi \, dx\, dt\\
\notag & \quad +   \tfrac{2}{R} \int_{T_R}\nabla\Delta u \cdot \nabla \varphi \, (\Sol u)\, B  \, dx \,dt \\
\notag &\quad +  \tfrac{1}{R}\int_{T_R} \Delta u\,  (\Sol u) \left[\Delta B\,\varphi + 2 \nabla B\cdot\nabla\varphi + B\Delta\varphi\right] \,dx\, dt \hspace{-1em}\\
\notag & \quad +  \tfrac{1}{R}\int_{T_R} \Delta u\cdot (\Sol u) \Delta B\, \varphi \, dx\, dt \\
\notag &\quad +  \tfrac{1}{R}\int_{T_R} \nabla u \cdot \frac{x}{(-4t)} \,B \, (\Sol u)\, \varphi \, dx\, dt\\
\notag & \quad +  \tfrac{1}{R} \int_{T_R} \nabla u \cdot \nabla\varphi \,\Delta B \,(\Sol u)\, dx \,dt,
\end{align}
where we have inserted  $\nabla_i\lap B = \frac{x}{4(-t)}B$ from Lemma~\ref{lem:propbiheat} in  the last step as well. Since $t<t_0=0$, $\partial_t u =- \Delta^2 u$ and by summing the first and the $6$th term we get that
\begingroup\allowdisplaybreaks\begin{align*}
1\text{st integral} + 6\text{th integral} &=  \tfrac{1}{4R} \int_{T_R} \frac{4 t \partial_ t u}{|t|}  \cdot   (\Sol u) B\,\varphi\, dx \,dt  \\
					  &\quad + \tfrac{1}{R}\int_{T_R} \nabla u \cdot \frac{x}{4|t|}  \, (\Sol u)\,B\, \varphi \, dx\, dt\\
& =  \tfrac{1}{R} \int_{T_R} \frac{(x\cdot \nabla u + 4 t \partial_ t u)^2}{4|t|}  B\,\varphi\, dx\, dt \\
& = \tfrac{1}{R} \int_{T_R} \frac{|\Sol u|^2}{4|t|} B\,\varphi\, dx\, dt.
\end{align*}\endgroup
Rearranging all terms, we end up with 
\begin{align}
\tfrac{\del}{\del R} \brk{\Psi \prs{u,R,\vp}} 
&=  \tfrac{1}{ R} \int_{T_R} \tfrac{ \brs{\Sol u}^2}{ 4\brs{t}} \cB \vp \,dx\,dt \label{equ:K1} \\
&\hsp  + \tfrac{2}{R} \int_{T_R} \prs{\N_j \lap u} \prs{\Sol u}  \brk{\N_j\cB} \vp \,dx\,dt \label{equ:K2} \\
&\hsp  + \tfrac{2}{R} \int_{T_R} \prs{\N_j \lap u} \prs{\Sol u} \cB  \brk{\N_j \vp} \,dx\,dt \label{equ:K3} \\ \displaybreak[0]
&\hsp + \tfrac{2}{R} \int_{T_R} \prs{\lap u} \prs{\Sol u}  \brk{\lap\cB}  \vp \,dx\,dt \label{equ:K4} \\ \displaybreak[0]
&\hsp +  \tfrac{2}{R} \int_{T_R} \prs{\lap u} \prs{\Sol u}  \N\cB \cdot \N  \vp \,dx\,dt \label{equ:K5} \\ \displaybreak[0]
&\hsp +  \tfrac{1}{R} \int_{T_R} \prs{\lap u} \prs{\Sol u}  \cB \, \lap \vp \,dx\,dt \label{equ:K6} \\
&\hsp +  \tfrac{1}{R} \int_{T_R} \prs{\N_i u} \prs{\Sol u}  \lap\cB \, \N_i \vp \,dx\,dt \label{equ:K7}.
\end{align}

We choose a cut-off function being positive around $x_0=0$ where $B$ is positive, then the very first integral (\ref{equ:K1}) is positive for non-shrinkers. This term can be used to compensate other (small) negative integrals including  highest-order derivatives on $u$. This representation above is favourable for us because the only integral where the highest derivatives (four spatial ones or one time-derivative) fall onto both of the $u$'s in the $L^2$-inner product is the first integral -- which has a sign. In all other integrals, we have a lower order derivative on $u$ in the inner product with $\Sol u$ (multiplied with the terms involving $B$ and $\vp$ which will not harm). Thus, we will use Young's inequality in each of these integrals. Let us use the abbreviation 
\[
 K_1 := \tfrac{1}{ R} \int_{T_R} \tfrac{ \brs{\Sol u}^2}{4 \brs{t}} B \vp \,dx\,dt
\]
for the integral in (\ref{equ:K1}). Next we use $\vp = \phi^4$, so we have $\N \vp = 4 \phi^3 \prs{\N \phi}$, and $\phi^3 = \phi^{4/2 + 2/2}$. We will also choose $\operatorname{spt}\phi$ in a domain such that $B\phi\geq 0$. We begin to estimate patiently
\begin{align*}
(\ref{equ:K2}) &= \tfrac{2}{R} \int_{T_R} \prs{\N_j \lap u} \prs{\Sol u}  \brk{\N_jB} \vp \,dx\,dt \\
&\leq \ge K_1 +  \tfrac{C (\ge)}{R} \int_{T_R} \brs{\N\lap u}^2 \tfrac{\brs{\N B}^2}{\brs{B}} \brs{t} \brs{\vp} \,dx\,dt\\ \displaybreak[0]
(\ref{equ:K3}) & = \tfrac{2}{R} \int_{T_R} \prs{\N_j \lap u} \prs{\Sol u} B  \brk{\N_j \vp} \,dx\,dt\\
&\leq \ge K_1 + \tfrac{C (\ge)}{R} \int_{T_R} \brs{\N\lap u}^2 |B||t|\,  |\N\phi|^2 \, \phi^2 \,dx\,dt.
\end{align*}
We continue with 
\begin{align*}
 (\ref{equ:K4}) & = \tfrac{2}{R} \int_{T_R} \prs{\lap u} \prs{\Sol u}  \brk{\lap B}  \phi^4 \,dx\,dt\\
 &\leq \ge K_1 + \tfrac{C(\ge)}{R}\int_{T_R} \brs{\lap u}^2 \tfrac{\brs{\lap B}^2}{|B|}|t|  \, \phi^4 \,dx\,dt
 \end{align*}
 and
 \begin{align*}
  (\ref{equ:K5}) & =\tfrac{8}{R} \int_{T_R} \prs{\lap u} \prs{\Sol u}  \N B \cdot \N  \phi \,\phi^3 \,dx\,dt\\
  &\leq \ge K_1 +  \tfrac{C(\ge)}{R}\int_{T_R} \brs{\lap u}^2 \tfrac{\brs{\N B}^2}{|B|}|t| \,|\N\phi|^2  \, \phi^2 \,dx\,dt.
\end{align*}
We also compute $\lap (\phi^4) = 4\operatorname{div} \left(\phi^3 \N\phi\right) = 12 \phi^2 |\N\phi|^2 + 4\phi^3\lap \phi $ and implement this in the computations for (\ref{equ:K6}):
\begin{align*}
 (\ref{equ:K6}) & = \tfrac{1}{R} \int_{T_R} \prs{\lap u} \prs{\Sol u}  B \, \lap \vp \,dx\,dt\\
 &\leq \ge K_1 +  \tfrac{C(\ge)}{R} \int_{T_R} \brs{\lap u}^2 |B| \, |t|\left(\brs{\N\phi}^4 + \phi^2 \brs{\lap \phi}^2\right) \,dx\,dt.
\end{align*}
We further compute 
\begin{align*}
  (\ref{equ:K7}) & =\tfrac{1}{R} \int_{T_R} \prs{\N_i u} \prs{\Sol u}  \lap B \, \N_i \vp \,dx\,dt \\
  &\leq \ge K_1 + \tfrac{C(\ge)}{R} \int_{T_R} \brs{\N u}^2 \tfrac{\brs{\lap B}^2}{\brs{B}} |t|\,\brs{\N \phi}^2 \phi^2 \,dx\,dt.
\end{align*}
The aim is now to study the growth condition of the biharmonic heat kernel in the integrals to get bounds. In related work for second order parabolic equations, upper bounds like
\begin{align*}
 \tfrac{|x|^2}{R}\vp \tilde k \leq C(1+ \tilde k) \ \ \ \text{ and } \ \ \tfrac{|x|^4}{R|t|} \vp \tilde k\leq C(1+\tilde k) \ \text{ on } T_R
\end{align*}
are used to continue estimates on $\tfrac{\partial}{\partial R}\Psi(u, R, \vp)$, see for example \cite[p.~451]{HongTian2004} or \cite{CS1}. Here, $\tilde k =\tfrac{c}{(-t)^{n/2}} \exp\left(-\tfrac{|x|^2}{4(-t)}\right)$ is the backwards heat kernel on $\AR^n$. Although looking like a technical detail, these estimates are crucial to get a monotonicity formula, in particular if one seeks such an inequality for all dimensions. In our case we struggle to get such bounds for all dimensions due to our cut-off function concentrating around $x_0$. We replace these uniform estimates by bounds depending on $R$. 

Let us explain some fundamental principles first. Due to the beautiful form (\ref{defn:B}) the blow-up behavior for $t\to 0^-$ of the kernel $B$, and also $\N B$ and $\lap B$ is well-understood. Since 
$|t|$ is comparable to $R^4$ on $T_R$, we write \\ $|t| \simeq R^4$ and analogously $\psi(x, |t|, R) \simeq R^{q}$ to express upper and lower bounds $cR^q \leq  |\psi| \leq C R^q$. We recall (see Lemma~\ref{lem:propbiheat}) on $T_R$
\begin{align}
\eta = \eta(x,t) & = \tfrac{|x|}{|t|^{1/4}} \simeq \tfrac{|x|}{R} \nonumber \\
 B(x,t) & = \tfrac{\alpha_n}{|t|^{n/4}} f_n (\eta) \simeq \tfrac{1}{R^n} \ \ \text{ with a singularity only at } R=0 \label{equ:growthB}\\
 \N B(x,t) & = \tfrac{\alpha_n \, x}{|x||t|^{(n+1)/4}} f_n'(\eta) \simeq \tfrac{1}{R^{n+1}} \label{equ:growthgrad}\\
 \lap B(x,t) & = \tfrac{1}{|t|^{(n+2)/4}}\lap^\eta f_n(\eta) \simeq \tfrac{1}{R^{n+2}}.\, \label{equ:growthlap}
\end{align}
Using $f_n'(\eta) = - \eta f_{n+2}(\eta)$ we compute in more detail
\begin{align*}
 \frac{|\N B|^2}{|B|^2} &= \frac{ | \alpha_n \eta f_{n+2}(\eta) \tfrac{x}{|x|\,|t|^{(n+1)/4}}|^2}{\tfrac{|\alpha_n|^2}{|t|^{n/2}} |f_n(\eta)|^2} \simeq \frac{\eta^2 |f_{n+2}|^2}{R^2 |f_n|^2} \ \ \ \text{ and }\\
  \frac{|\lap B|^2}{|B|^2}&=\frac{\left|\tfrac{1}{|t|^{1/2}} \lap^\eta f_n(\eta)\right|^2}{|f_n(\eta)|^2} = \frac{\left|\lap^\eta f_n(\eta)\right|^2}{|t|\,|f_n(\eta)|^2} \simeq \frac{\left|\lap^\eta f_n(\eta)\right|^2}{R^4 |f_n(\eta)|^2}.
\end{align*}
We will choose our cut-off function such that
\[
 C_1\geq \tfrac{|f_{n+2}(\eta)|}{|f_n(\eta)|}\geq C_0 >0 \ \ \text{ and } \ \   C_1\geq\tfrac{|\lap^\eta f_n(\eta)|}{|f_n(\eta)|}\geq C_0 >0.
\]
We first collect additional conditions on $\phi$ by continuing the estimates and do the explicit definition of $\phi$ afterwards.\\

We use the growth of the kernel $B$ and its derivative in the following estimates. In order to bound (\ref{equ:K2}) we combine (\ref{equ:growthB}), (\ref{equ:growthgrad}), $|\phi|\leq 1$ and $|t|\simeq R^4$ to get the existence of a constant $C>0$ only depending on $f_n$ such that
\begin{align} \label{equ:new1}
 \tfrac{|\nabla B|^2\,|\phi|\,|t|}{|B| R } \leq \tfrac{C}{R^{n-1}} \ \ \ \text{ on } T_R.
\end{align}
Using (\ref{equ:new1}) we get that
\label{pageproof}
\begin{align*}
 (\ref{equ:K2}) & \leq \ge K_1 + C R^{1-n} \int_{T_R}  \brs{\N\lap u}^2  \,dx\,dt\\
 & = \ge K_1 +  C R^{1-n} \int_{-16 R^4}^{-R^4} \int_{\AR^n}\brs{\N\lap u}^2  \,dx\,dt\\
 &\leq \ge K_1 + C R^{1-n} \int_{-\hat T}^0 \int_{\AR^n}|\nabla\lap u|^2 \,dx\, dt\\
 & \leq  \ge K_1 +  C R^{1-n} \int_{\AR^n}|\nabla u_0|^2 dx.
\end{align*}
We define  $\phi(x,t) := \phi_0\left(\tfrac{|x|}{(-t)^{1/4}}\right) = \phi_0(\eta)$, where $\phi_0$ is a cut-off function for $B(x,-1)$ satisfying $\nabla\phi_0= 0$ on $B_{\rho}(0)$ for $\rho$ depending only on $f_n$. We get that
\[\nabla \phi(x,t) = \frac{x}{|x|(-t)^{1/4}} \nabla \phi_0\left(\tfrac{|x|}{(-t)^{1/4}}\right)\]
implying 
\[
|\nabla \phi (x,t)| \simeq \frac{1}{R} \left|\nabla \phi_0\left(\eta\right) \right| \ \ \text{ and }\ \ |\Delta \phi (x,t)| \simeq \frac{1}{R^2} |\Delta \phi_0 (\eta) + \tfrac{n-1}{\eta}\nabla\phi_0(\eta)|.
\]
This leads to 
\begin{align*}
 \tfrac{| B|\,|\nabla\phi|^2\,|t|}{R } \leq C R^{1-n}\ \ \ \text{ on } T_R
\end{align*}
for a constant independent of $R$. We achieve
\begin{align*}
 (\ref{equ:K3}) & \leq \ge K_1 + C R^{1-n} \int_{T_R}  \brs{\N\lap u}^2  \,dx\,dt \\
 &\leq \ge K_1 + C R^{1-n}\int_{\AR^n}|\nabla u_0|^2 dx
\end{align*}
by (\ref{est:N-2}). 
We also use the exponential decay of $f_n$ for $\eta$ large and \\ $\eta \simeq \frac{|x|}{R} \geq \rho>0$  on the support of $\phi$ in the case where $\Delta \phi$ is involved to get bounds
\begin{align*}
  \tfrac{|\Delta B|^2\,|\phi|^4\,|t|}{|B| \, R } & \leq C R^{-1-n} \ \ \ \text{ and }  \ \ \  \tfrac{|\nabla B|^2\,|\nabla\phi|^2\,|\phi|^2\,|t|}{|B| \, R }  \leq C R^{-1-n} \ \ \ \text{ on } T_R,\\
  \tfrac{|B|\,|t|\,(|\nabla\phi|^4+ |\phi|^2 |\Delta\phi|^2)}{ R }  &\leq C R^{-1-n} \ \ \ \text{ and }  \ \ \  \tfrac{|\Delta B|^2\,|\nabla\phi|^2\,|t|}{|B| \, R }  \leq C R^{-3-n} \ \ \ \text{ on } T_R
\end{align*}
which we incorporate into (\ref{equ:K4}),  (\ref{equ:K5}), (\ref{equ:K6}) and (\ref{equ:K7}). We also use
\[
 \int_{T_R}|\nabla u|^2 \,dx\, dt \leq c R^4 \sup_{t}\int_{\AR^n}|\nabla u|^2 dx \leq c R^4 \int_{\AR^n}|\nabla u_0|^2 dx
\]
whenever necessary (analogously for $|\Delta u|^2$)
from (\ref{est:N-2}), (\ref{est:lap-2}). 
We arrive at
\begin{align*}
 (\ref{equ:K4}) + (\ref{equ:K5})  &\leq 2 \ge K_1 + C R^{3-n}\int_{\AR^n} \brs{\lap u_0}^2 \,dx, \\
 (\ref{equ:K6}) + (\ref{equ:K7}) &\leq 2 \ge K_1 +
C R^{3-n} \int_{\AR^n} |\lap u_0|^2 dx + C R^{1-n}\int_{\AR^n} |\nabla u_0|^2 dx.
\end{align*}
Finally, we sum up all terms, we choose $\ge = \tfrac{1}{16}$ implying $1-8\ge = \tfrac{1}{2}$ to get
\begin{align} \begin{split}
\tfrac{\del}{\del R} \brk{\Psi \prs{u,R,\vp}}
& \geq \tfrac{1}{2}\tfrac{1}{ R} \int_{T_R} \tfrac{ \brs{\Sol u}^2}{4 \brs{t}} B \vp \,dx\,dt  - C  R^{1-n} \int_{\AR^n}|\N u_0|^2 dx\\
&\hsp - C  R^{3-n} \int_{\AR^n}|\lap u_0|^2 dx. \label{mainInequ} \end{split}
\end{align}

We need to show that a cut-off function $\phi$ with the desired properties exists. We collect the conditions that $\phi$ needs to satisfy:
\begin{enumerate}
 \item $B\phi\geq 0$ and $\Delta B \phi\leq 0$ on $\operatorname{spt} \phi$ , \label{cut-off1}
 \item \label{cut-off3} $C_1 \geq \frac{|f_{n+2}(\eta)|}{|f_n(\eta)|}\geq C_0>0$,
 \item $C_1 \geq\frac{\left| \Delta^\eta f_n(\eta)\right|}{|f_n(\eta)|} \geq C_0>0$,\label{cut-off4}
 \item the usual properties like smoothness, $0\leq \phi\leq 1$, $\operatorname{spt} \phi$ compact. \label{cut-off5}
\end{enumerate}
We first consider $t= -1$ for the kernel $B(x,-1)= f_n(|x|) = f_n(r)$. We know that $f_n$ attains its maximum at $r=0$, and the value is positive \cite{GG}. We choose $\Tilde{r}_{1}>0$ such that $f_n(r)\geq c>0$ for all $r\leq\Tilde{r}_{1}$. We can, in fact, make a quantitative statement about $\Tilde{r}_{1}$ and $c$, see Lemma~\ref{lem:bound2}~i). We use Lemma~\ref{lem:bound1} which shows that the first zero of $f_{n+2}$ is larger than the first zero of $f_n$. This implies $C_1 \geq \frac{|f_{n+2}(r)|}{|f_n(r)|} = \frac{f_{n+2}(r)}{f_n(r)}\geq C_0>0$ for $r\leq \Tilde{r}_{1}$, also using that $f_n, f_{n+2}$ are bounded functions. \\ 
From \cite[Proof of Theorem~1, p.~2970]{GG} we know that $\Delta^r f_n(0) <0$ and $\Delta^r f_n$ is strictly increasing on each interval where $f_n \geq 0$, in particular on the first one which starts in $r=0$.  We  estimated the location of the first zero of $\Delta^r f_n$ in Lemma~\ref{lem:bound2} ii). From that, we get an $\tilde r_2>0$ such that $\Delta^r f_n(r)\leq -c <0$ for $r\leq \tilde r_2$. We take $r_2 := \min (\Tilde{r}_{1}, \tilde r_2) $.
This means that we can choose the cuf-off function $\phi_0\in C^{\infty}_c$ with $\phi_0=\phi_0(|x|) =\phi_0(r)$ to be a smooth function satisfying
\begin{align*}
\phi_0 (r) = \begin{cases}
             1 & \text{ on } B_{\frac{r_2}{4}}(0)\\
            0 & \text{ on } \AR^n\setminus B_{\frac{r_2}{2}}(0).
           \end{cases}
\end{align*}
and a smooth interpolation in between. This $\phi_0$ satisfies all requirements (\ref{cut-off1}) to (\ref{cut-off5}) for $t=-1$. Now we recall
\begin{align*}
\phi(x,t) = \phi_0(\eta) = \phi_0\left(\frac{|x|}{(-t)^{\frac{1}{4}}}\right).
\end{align*}
By $B(x,t)= \frac{\alpha_n}{(-t)^{\frac{1}{4}}} f_n(\eta)$ and $\Delta B(x,t) = {(-t)^{-\frac{n+2}{4}}}\Delta^{\eta}f_n(\eta)$ all properties carry over to $\phi(x,t)$. We also have ensured that $\frac{d}{dR} \varphi_R =0$. In this way, $\phi(x,t)$ only depends on $f_n(\eta)$.
\end{proof}

 A more desirable result than Theorem~\ref{mainthm-derivatives-withoutSUP} is a monotonicity formula where the power of the polynomial growth is positive for more dimensions than $n=1$ or where it is even independent of $n$. For the linear case, we are able to improve the power of one term slightly by modifying our proof. The price is that the term $\sup_t \int|\nabla\Delta u(\cdot, t)|^2 dx$ comes into play. This is formulated in the following results. 
 
 \begin{thm}\label{mainthm-derivatives}
  Let $u:\AR^n \times [0,\infty) \to\AR$ be a smooth solution of the bi-heat equation on $\AR^n$, $n\leq 4$, with initial datum $u_0:\AR^n \to \AR$ satisfying $ \int_{\AR^n} |\nabla^l u (x, t)|^2 dx<\infty$ for each $t\in [0,\infty)$ and $l=1,...,4$. Then, for any $(x_0,t_0)\in\AR^n \times (0,\infty)$ there is a smooth cut-off function \\ $\varphi \in C^\infty_c (\AR^n\times[0,\infty))$ depending on $f_n$ and a constant $C=C(f_n)$ such that, for $0< R \leq t_0^{1/4}/2$ we have
\begin{align} 
\label{mainInequ-0} 
\tfrac{\del}{\del R} \brk{\Psi \prs{u,R,\vp}} 
& \geq \tfrac{1}{2}\tfrac{1}{ R} \int_{T_R(t_0)} \tfrac{ \brs{\Sol u}^2}{4 \brs{t-t_0}} B \vp \,dx\,dt  \\
\notag &\hsp - C  R^{5-n} \int_{\AR^n}|\N \Delta u_0|^2 \,dx\\
\notag &\hsp - C R^{1-\frac{n}{2}} \left(\int_{\AR^n}|\nabla u_0|^2 \,dx + \int_{\AR^n}|\lap u_0|^2 \,dx\right) \\
\notag &\hsp- C R^{3-n}  \int_{\AR^n}|\lap u_0|^2 \,dx.
\end{align}
Furthermore, the quantity $\Phi=\Phi(u, R,\varphi)$ on the time-slice $S_R(t_0)$ satisfies the analogous inequality

\begin{align}  \label{mainInequ-0-forPhi}
\tfrac{\del}{\del R} \brk{\Phi \prs{u,R,\vp}} 
& \geq \tfrac{1}{8 R}\int_{S_R(t_0)} \brs{\Sol u}^2 B \vp \,dx  \\
	\notag &\hsp - C R^{5-n} \int_{\AR^n}|\N \Delta u_0|^2 \,dx\\
\notag &\hsp - C R^{1-\frac{n}{2}} \left(\int_{\AR^n}|\nabla u_0|^2 \,dx + \int_{\AR^n}|\lap u_0|^2 \,dx\right)  \\
\notag &\hsp- C R^{3-n} \int_{\AR^n}|\lap u_0|^2 \,dx.
\end{align}
 \end{thm}

\begin{proof}  We repeat the proof of Theorem~\ref{mainthm-derivatives-withoutSUP} until page~\pageref{pageproof} where we estimated (\ref{equ:K2}). Instead of (\ref{est:N-2}) we use (\ref{est:Nlap-2}): 
\begingroup\allowdisplaybreaks\begin{align*}
 (\ref{equ:K2}) & \leq \ge K_1 + C R^{1-n} \int_{T_R}  \brs{\N\lap u}^2  \,dx\,dt\\
 & = \ge K_1 + C R^{1-n} \int_{-16 R^4}^{-R^4} \int_{\AR^n}\brs{\N\lap u}^2  \,dx\,dt\\
 &\leq \ge K_1 + C R^{5-n} \sup_{t}\int_{\AR^n}\brs{\N\lap u}^2 (x,t) \,dx\\
 &\leq \ge K_1 + C  R^{5-n} \int_{\AR^n} |\N\lap u_0|^2\,dx .
\end{align*}\endgroup
We also estimate (\ref{equ:K7}) slightly differently. From the previous proof we arrive at
\begin{align*}
(\ref{equ:K7}) \leq C R^{1-n} \sup_{t\in [-16 R^4, -R^4]} \int_{[\phi>0]} |\nabla u|^2 dx.
\end{align*}
Due to H\"older inequality, Sobolev embedding $W^{1,2}\hookrightarrow L^4 $, integration by parts, the fact that $\operatorname{spt}(\phi)\subset B_{\rho R}(0)$ and $\int |\nabla u|^2 \leq \int|\nabla u_0|^2$ from (\ref{est:lap-2}) we get that
\begin{align*}
 (\ref{equ:K7}) \leq C R^{1-\frac{n}{2}} \left(\int|\nabla u_0|^2 dx + \int|\Delta u_0|^2 dx\right)  
\end{align*}
since $|B_{\rho R}(0)| = C R^{\frac{n}{2}}$.
With all other terms we proceed as in the proof of Theorem~\ref{mainthm-derivatives-withoutSUP} and arrive at~(\ref{mainInequ-0}). \\
For the time-slice version with $S_R$ instead of $T_R$ we can use the same proof. Note that $|t|=R^4$ in this case which also results in the extra $R^4$ in front of the definition of $\Phi$ to make it scaling invariant.   We only have to be careful with estimating the analogous terms of (\ref{equ:K2}) and (\ref{equ:K3}) where we cannot use (\ref{est:N-2}) but (\ref{est:Nlap-2}) instead. Above, we already treated (\ref{equ:K2}) appropriately. For the analogue of (\ref{equ:K3}), we simply use (\ref{est:Nlap-2}) directly:
\begin{align*}
&2 R^3 \int_{S_R} (\nabla_j\Delta u) (\Sol u) (\nabla_j B) \varphi \, dx \\
&\qquad  \leq \ge \frac{R^3}{4|t|}\int_{S_R}(\Sol u)^2 B \phi^4 \, dx + C R^{4} \int_{S_R} |\N\lap u|^2 \,dx\,dt  \\
&\qquad \leq   \frac{\ge}{4 R}\int_{S_R}(\Sol u)^2 B \phi^4 \, dx  + C R^{4} \int_{\AR^n}|\N \Delta u_0|^2 \,dx. \\[-4em]
\end{align*}
\end{proof}
 

  \begin{cor}\label{cor:mainresult}
  Let $u:\AR^n \times [0,\infty) \to\AR$, a smooth solution of the bi-heat equation on $\AR^n$, $n\leq 4$, with initial datum $u_0:\AR^n \to \AR$ and bounds  $ \int_{\AR^n} |\nabla^l u (x, t)|^2 dx<\infty$ for each $t\in [0,\infty)$ and $l=1,...,4$. Then, for any $(x_0,t_0)\in\AR^n \times (0,\infty)$ there is a smooth cut-off function \\ $\varphi \in C^\infty_c (\AR^n \times[0,\infty))$ depending on $f_n$ and a constant $C=C(f_n)$ such that, for $0< R_1 <R_2 \leq  t_0^{1/4}/2$ the following estimates hold
  \begin{enumerate}
  \item if $1\leq n\leq  3$, $\Psi$ is a pre-entropy with inequality
  \begin{align*} 
   \Psi( u, R_1, \vp) &\leq \Psi (u, R_2,\vp) + C\left(R_2^{6-n}- R_1^{6-n}\right)\int_{\AR^n}|\N \Delta u_0|^2 \,dx \\
  \notag  &\hsp + C \left(R_2^{\frac{4-n}{2}}- R_1^{\frac{4-n}{2}}\right)    \left( \int_{\AR^n}|\N u_0|^2 \,dx + \int_{\AR^n}|\Delta u_0|^2 \,dx \right)  \\
  &\hsp + C\left(R^{4-n}_2-R_1^{4-n}\right) \int_{\AR^n} |\lap u_0|^2 \,dx.
  \end{align*}
  \item if $n=4$, then we get 
  \begin{align*}
   \Psi( u, R_1, \vp) &\leq \Psi (u, R_2,\vp) + C\left(R_2^2-R_1^2\right) \int_{\AR^n}|\N \Delta u_0|^2 dx \\
   &\hsp + C(\log(R_2)- \log(R_1)) \left( \int_{\AR^n}|\N u_0|^2 \,dx + \int_{\AR^n}|\Delta u_0|^2 \,dx \right)  \\
   &\hsp + C  (\log(R_2)- \log(R_1)) \int_{\AR^n} |\lap u_0|^2 dx.
  \end{align*}

  \end{enumerate}

 \end{cor}

\begin{proof} We integrate (\ref{mainInequ-0}) and need to be careful in the case $n=4$. 
\end{proof}
 
\subsection{The nonlinear case} \label{sect:nonlinear}

We first derive some global a priori estimates for the extrinsic biharmonic map heat flow following \cite{LammExtr}.
\begin{lemma} \label{lem:apriori}
Let $N$ be a closed, smooth manifold isometrically embedded into $\AR^k$.
 Let $u:\AR^n \times [0,T) \to N\hookrightarrow \AR^k$ be a smooth solution of the extrinsic biharmonic map heat flow on $\AR^n$ with smooth initial datum \\ $u_0:\AR^n \to N \hookrightarrow \AR^k$  satisfying
 \begin{align*}
 \int_{\AR^n}|\nabla^l u|^2 dx <\infty \qquad\qquad \text{ for each } t\in [0,T) \text{ and } l\in\{1,2,3,4\},
 \end{align*}
then the following inequalities hold true:
 \begin{enumerate}
     \item $\frac{1}{2}\partial_t \int_{\AR^n}  |\Delta u|^2 dx + \int_{\AR^n} |\partial_t u|^2 \leq 0$, \label{first-1}
     \item $\int_{\AR^n}|\nabla^2 u|^2dx \leq 4 \int_{\AR^n}|\Delta u_0|^2$, \label{second-2}
     \item If also $\sup_{t\in[0,T)}\int_{\AR^n}|u|^2\leq c_1^2$, then
 \begin{align*}
 \int_{\AR^n}|\nabla u|^2 \leq 2 c_1 \left(\int_{\AR^n}|\Delta u_0|^2\right)^{\frac{1}{2}},
 \text{ and if } n\leq 4, \text{ then also }\int_{\AR^n}|\nabla u|^4\leq c_2,
 \end{align*}
  where $c_2$ depends on $c_1$ and on $\int_{\AR^n}|\Delta u_0|^2$.  
    \label{second-zw-2u3}
     \item If $n\in\{1,2,3\}$ and $\sup_{t\in[0,T)}\int_{\AR^n}|u|^2\leq c_1^2$, then there is a constant \\ $c>0$ depending on $c_1$, $N$,  and on $\int_{\AR^n}|\Delta u_0|^2$ such that \\ $\frac{1}{2}\partial_t\int_{\AR^n}|\nabla u|^2 dx + \frac{1}{2}\int_{\AR^n}|\nabla^3 u|^2dx \leq c \left(1 +\int_{\AR^n}|\partial_t u|^2 \right),$ \label{third-3}
     \item If $n=4$ and $\sup_{t\in[0,T)}\int_{\AR^n}|u|^2\leq c_1^2$, then there is an $\epsilon_0>0$ such that if $\int_{\AR^n}|\nabla u|^4dx <\epsilon_0$, then \\
     $\frac{1}{2}\partial_t\int_{\AR^n}|\nabla u|^2 dx + \frac{1}{2}\int_{\AR^n}|\nabla^3 u|^2dx \leq c \left(1 +\int_{\AR^n}|\partial_t u|^2 \right)$  for a constant $c>0$ depending on $c_1$, $N$,  and on $\int_{\AR^n}|\Delta u_0|^2$. \label{fourth-4}
 \end{enumerate}
\end{lemma}

\begin{proof}
This proof is based on \cite[Section~2]{LammExtr}, but it is slightly adapted to the non-compact base manifold. In all integrals we integrate over ${\AR^n}$ if not indicated otherwise.
Let $\psi_r$ be as in the proof of Lemma~\ref{lem:flowprop}, thus $\psi_r$ is a smooth cut-off function with $\psi_r=1$ on $B_{\frac{r}{2}}(0)$ and $\operatorname{spt}(\psi_r)\subset B_r(0)$. We compute
\begin{align*}
\frac{1}{2}\partial_t\int|\Delta u|^2 \psi_r^4 & = \int\langle \Delta \partial_t u,\Delta u\rangle \psi_r^4\\
&= \int \langle \partial_t u, \Delta^2 u\rangle \psi_r^4 + \int\langle \partial_tu,\Delta u\rangle (12\psi_r^2|\nabla\psi_r|^2 + 4\psi_r^3 \Delta\psi_r)  \\ \displaybreak[0]
&\quad + 8\int\langle\partial_t u,\nabla_i \Delta u\rangle \psi_r^3 \nabla_i\psi_r\\  
& \leq -\int|\partial_t u|^2 \psi_r^4 + c \int |\partial_t u||\Delta u| (\psi_r^2|\nabla\psi_r|^2 + \psi_r^3|\Delta\psi_r|)\\
&\quad + 8\int |\partial_t u||\nabla \Delta u| \psi_r^3 |\nabla\psi_r|\\ \displaybreak[0]
&\leq -(1-2\epsilon) \int|\partial_t u|^2 \psi_r^4 + c\int|\Delta u|^2 (|\nabla \psi_r|^4 + |\Delta \psi_r|^2) \\
&\quad + c\int|\nabla\Delta u|^2 |\nabla \psi_r|^2,
\end{align*}
where we have used the fact that $\int\langle f(u),\phi\rangle=0$ for $\phi\in C^\infty_c(\AR^n,\AR^k)$ with $\phi(x)\in T_{u(x)}N$ and Young's inequality. Now, we first let $r$ go to infinity and then $\epsilon$ go to zero and get $(\ref{first-1})$. For $(\ref{second-2})$, we integrate by parts
\begin{align}
 \label{inequ13}
\int|\nabla^2 u|^2 \psi_r^2 & = - \int\langle \nabla_k u, \Delta\nabla_k u\rangle \psi_r^2 - 2\int \psi_r\nabla_i \psi_r \langle\nabla_ku, \nabla_i\nabla_k u\rangle\\
\notag &= \int |\Delta u|^2 \psi_r^2 + 2\int\langle \nabla_i u, \Delta u\rangle \psi_r\nabla_i\psi_r \\
\notag &\quad -  2\int \psi_r\nabla_i \psi_r \langle\nabla_ku, \nabla_i\nabla_k u\rangle\\
\notag & \leq 2 \int |\Delta u|^2 \psi_r^2 + c \int|\nabla u|^2 |\nabla \psi_r|^2 + \frac{1}{2}\int|\nabla^2 u|^2 \psi_r^2
\end{align}
which implies we can absorb the Hessian-term on the left hand side. Using again $r\to\infty$ and (\ref{first-1}) we get (\ref{second-2}). For the first term in (\ref{second-zw-2u3}), we integrate by parts, use (\ref{first-1}) and H\"older inequality to get
\begin{align*}
\int|\nabla u|^2 \psi_r^2 & = - \int\langle \Delta u, u\rangle \psi_r^2 - 2 \int \langle u,\nabla_i u\rangle \psi_r^2 \nabla_i \psi_r \\
& \leq \|u\|_{L^2(\AR^n)} \left(\int|\Delta u|^2\right)^{\frac{1}{2}} + \frac{1}{2} \int|\nabla u|^2\psi_r^2 + c \int |u|^2 |\nabla\psi_r|^2\\
&\leq c_1 \left(\int|\Delta u_0|^2\right)^{\frac{1}{2}} + \frac{1}{2} \int|\nabla u|^2\psi_r^2 + c \int |u|^2 |\nabla\psi_r|^2.
\end{align*}
We absorb the second term to the left hand side and let $r\to \infty$ to get the desired inequality.

For the second inequality in (\ref{second-zw-2u3}), we use the Sobolev embedding and (\ref{first-1}) and $\int|\nabla u|^2\leq c$.

Since (\ref{fourth-4}) and (\ref{third-3}) are proven in a similar way (with appropriate Sobolev embeddings), we restrict ourselves to the proof of  (\ref{fourth-4}).
For $n=4$, we use the Sobolev embedding $W^{1,2}\hookrightarrow L^4 $ and get
\begingroup\allowdisplaybreaks\begin{align}
 \label{inequ14}
\frac{1}{2}\partial_t &\int|\nabla u|^2 \psi_r^4 = \int\langle \nabla \partial_t u, \nabla u \rangle\psi_r^4\\
\notag &= -\int\langle \partial_t u,\Delta u\rangle \psi_r^4 - 4\int \psi_r^3 \nabla_i \psi_r \langle \partial_t u, \nabla_i u\rangle\\  
\notag &=-\int\langle \partial_t u, \Delta u + A(u)(\nabla u, \nabla u)\rangle \psi_r^4 - 4\int \psi_r^3 \nabla_i \psi_r \langle \partial_t u, \nabla_i u\rangle\\ 
\notag & = \int \langle \Delta^2 u, \Delta u + A(u)(\nabla u,\nabla u)\rangle \psi_r^4- 4\int \psi_r^3 \nabla_i \psi_r \langle \partial_t u, \nabla_i u\rangle\\
\notag &\leq - \int|\nabla\Delta u|^2\psi_r^4 + \int|\nabla\Delta u|(|\nabla u|^3 + |\nabla u||\nabla^2 u|) \psi_r^4 \\
\notag & \quad
- 4\int\psi_r^3 \nabla_i\psi_r\langle \nabla_i \Delta u, \Delta u + A(u)(\nabla u,\nabla u)\rangle\\
\notag &\quad - 4\int \psi_r^3 \nabla_i \psi_r \langle \partial_t u, \nabla_i u\rangle \\
\notag &\leq -(1-\epsilon)\int|\nabla\Delta u|^2\psi_r^4 +  c\int (|\nabla u|^6 + |\nabla u|^2|\nabla^2 u|^2) \psi_r^4 \\
\notag & \quad + c \int (|\Delta u|^2 + |\nabla u|^4 + |\nabla u|^2)|\nabla\psi_r|^2 + c\int|\partial_t u|^2 \psi_r^4.
\end{align}\endgroup
We further estimate as in (2.6) and (2.8) of \cite{LammExtr} (see also \cite[Section~2.3]{LammDiss})
\begin{align*}
\int|\nabla u|^2|\nabla^2 u|^2 \psi_r^4 &\leq c\left(\int|\nabla u|^4\right)^{\frac{1}{2}}\biggl(\int|\nabla^3 u|^2\psi_r^4 + \int|\nabla^2 u|^2\psi_r^4 \\
					&\hspace{8em}+ \int|\nabla\psi_r|^2 |\nabla^2 u|^2\biggr),\\
\int|\nabla u|^6 \psi_r^4 &\leq c\left(\int|\nabla u|^4\right)^{\frac{1}{2}}\left(\int|\nabla^2 u|^2|\nabla u|^2 \psi_r^4 + \int |\nabla\psi_r|^2 |\nabla u|^4\right).
\end{align*}
As in (\ref{inequ13}), we get $\int|\nabla^3 u|^2 \leq c \int|\nabla\Delta u|^2$. Combining these three estimates with (\ref{inequ14}) and choosing $\int|\nabla u|^4<\epsilon_0$ appropriately we arrive at
\begin{align*}
&\frac{1}{2}\partial_t \int|\nabla u|^2 \psi_r^4 + \frac{1}{2} \int|\nabla^3 u|^2 \\
&\qquad  \leq  c \int (|\Delta u|^2 + |\nabla u|^4 + |\nabla u|^2 +|\nabla^2 u|^2)|\nabla\psi_r|^2 \\
& \qquad\quad + c\int|\partial_t u|^2 \psi_r^4 + c\int|\nabla^2 u|^2 \psi_r^4.
\end{align*}
We use $\int|\nabla u|^4 + \int|\nabla^2 u|^2 \leq c$ and let $r\to\infty$ to get the result.
\end{proof}
From this lemma, we immediately get the following.

\begin{cor}\label{cor:W3est}
Let $N$ be a closed, smooth manifold isometrically embedded into $\AR^k$.
 Let $u:\AR^n \times [0,T) \to N\hookrightarrow \AR^k$, $n\leq 4$, be a smooth solution of the extrinsic biharmonic map heat flow on $\AR^n$ with smooth initial datum $u_0:\AR^n \to N \hookrightarrow \AR^k$  satisfying
 \begin{align*}
 &\int_{\AR^n}|\nabla^l u|^2 dx <\infty \qquad \text{ for each } t\in [0,T) \text{ and } l\in\{1,2,3,4\},\\
 &\sup_{t\in[0,T)} \int_{\AR^n}|u|^2 \leq c_1^2 <\infty  \qquad  \text{ and if } n=4: \int_{\AR^n}|\nabla u|^4 <\epsilon_0,
 \end{align*}
 with $\epsilon_0$ from Lemma~\ref{lem:apriori}. Then there is a constant $C>0$ depending on $c_1$ and on $N$  such that
 \begin{align*}
  \int_0^t \int_{\AR^n} |\nabla^3 u|^2 dx\, ds \leq C (1+t) \int_{\AR^n}|\Delta u_0|^2 + C\left(\int_{\AR^n}|\Delta u_0|^2\right)^{\frac{1}{2}}.
 \end{align*}
\end{cor}

In our estimates on the derivative of $\Psi$ we have the advantage that the support of $\varphi$ is already contained in a finite ball. We can use this to to remove the assumption of being globally in $L^2$ and to improve the above estimates.

\begin{prop}\label{prop:absch}
 Let $u:\AR^n \times [0,T] \to N\hookrightarrow \AR^k$, $2\leq n\leq 3$ and $T<\infty$, be a smooth solution of the extrinsic biharmonic map heat flow on $\AR^n$ with smooth initial datum $u_0:\AR^n \to N \hookrightarrow \AR^k$.
 Let also $(x_0, t_0)\in \AR^n\times (0,T]$ be arbitrary and $\varphi$ the cut-off function from Theorem~\ref{mainthm-derivatives-withoutSUP}. Then there is a constant $C(u_0)>0$ depending on $n$, $u_0$, $N$ and $T$ such that 
 \[
 \sup_{t\in [0,t_0]} \int_{\AR^n}|\nabla\Delta u|^2 \sqrt{\varphi}\, dx + \sup_{t\in [0,t_0]} \int_{\AR^n}|\nabla u|^2 \sqrt{\varphi}\, dx \leq C(u_0).
 \]
\end{prop}
\begin{proof}
Note that we have constructed $\varphi$ with support around $x_0$. More precisely, $\varphi$ satisfies by construction $spt(\varphi) \subset B_{\rho R}(x_0) \subset B_{\rho T^{\frac{1}{4}}}(x_0)$ for a constant $\rho >0$ only depending on $f_n$. Since $u$ maps into a compact manifold $N$, this particularly implies that 
\[
\sup_{t\in [0,T]}\int_{[ \varphi>0]}|u|^2 dx \leq c,
\]
where $c$ is a constant depending on $N$ and on the volume of the ball $B_{\rho T^{\frac{1}{4}}}(0)\subset \AR^n$. Repeating the strategy of the proof of (\ref{second-zw-2u3}) of Lemma~\ref{lem:apriori} by writing $\sqrt{\varphi}=\phi^2$ with $|\nabla \phi|\leq \frac{C}{R}$ as before, we get that
\begin{align*}
\int|\nabla u|^2 \phi^2 & = - \int\langle \Delta u, u\rangle \phi^2 - 2 \int \langle u,\nabla_i u\rangle \phi^2 \nabla_i \phi \\
& \leq c(N, n) \left(\int|\Delta u|^2\right)^{\frac{1}{2}} + \frac{1}{2} \int|\nabla u|^2\phi^2 + c \int |u|^2 |\nabla\phi|^2\\
&\leq c(N, n) \left(\int|\Delta u_0|^2\right)^{\frac{1}{2}} + \frac{1}{2} \int|\nabla u|^2\phi^2 + c(N) \int_{B_{\rho R}(x_0)} |\nabla\phi|^2\\
&\leq c(N, n) \left(\int|\Delta u_0|^2\right)^{\frac{1}{2}} + \frac{1}{2} \int|\nabla u|^2\phi^2 + c(N,n, T),
\end{align*}
where we used $n\geq 2$. This shows the second inequality. For the first, we use \cite{LammExtr} to show that in fact, all $C^{k,\alpha}$-norms of $u$ are bounded on compact subsets in terms of the $C^{k}$ norms of the initial datum $n$, $u_0$, $N$ and $T$ which implies the claimed estimate.\\
For that, we first observe that the short-time existence proof of the biharmonic map heat flow via parabolic Schauder-estimates provides \\
$\|u\|_{C^{k}(K)} \leq c \|u_0\|_{C^{k}(K)}$ on compact subsets $K\subset \AR^n$ for a short time $[0,\sigma]$, $\sigma>0$. Then it was shown in \cite[Theorem~3.8]{LammExtr} that there is a small $\delta>0$ such that, as soon as the flow has evolved this tiny bit, all $C^k$-norms on compact sets and on $[ t-\delta/4,  t]$ can be estimated in terms of the $C^{k}$-norm on $[ t -\delta/2,  t-\delta/4]$. In order to do that, we use the local $L^2$-bound on $u$ and consequently  $\int|\nabla u|^2\sqrt{\varphi}\leq c$ that we have already proved in the first step.
Then, all we need to do is to cover the compact interval $[0,t_0]$ with finitely many such $\delta/4$-intervals and we get a bound on the $W^{4,2}$-norm that in fact depends only on $n$, $u_0$ and $N$. From that, we continue as in \cite[Theorem~3.8]{LammExtr}. We emphasize why the bound does not depend on $R$: The cut-off function $\eta^4$ and $R$ in
\cite[Section~3]{LammExtr} plays the role of our $\sqrt{\varphi}$ and $R$. We fix a positive $R_0>0$ with $R\geq R_0$. In some estimates in the work of Lamm, $R$ appears in the denominator and the $\delta$ from \cite[Lemma~3.6]{LammExtr} is used to bound $t/R^4$ in estimates like (3.5) in \cite{LammExtr}. In these cases, we estimate $\frac{1}{R} \leq \frac{1}{R_0}$. In that way we get a (possibly small) $\delta>0$ such that we can cover the interval $[0,t_0]$ by finite subintervals, where we get estimates on the $W^{4,2}$-norm on each subinterval in terms of information of the previous subinterval. So, in the end, it only depends on $u_0, n, N$ and $T$.
\end{proof}

\begin{thm}\label{mainthm-generalN}
  Let $u:\AR^n \times [0,T) \to N\hookrightarrow \AR^k$ be a smooth solution of the extrinsic biharmonic map heat flow on $\AR^n$ with smooth initial datum \\ $u_0:\AR^n \to N \hookrightarrow \AR^k$. Then, for any $(x_0,t_0)\in\AR^n \times (0,T]$ there is a smooth cut-off function $\varphi \in C^\infty_c (\AR^n\times [0,T))$ depending on $f_n$ and a constant \\ $C=C(f_n)$ such that, for every $R$ with $0< R < t_0^{1/4}/2,$ we have
      \begin{align*} 
\tfrac{\del}{\del R} \brk{\Psi \prs{u,R,\vp}} 
& \geq \tfrac{1}{2}\tfrac{1}{ R} \int_{T_R(t_0)} \tfrac{ \brs{\Sol u}^2}{4 \brs{t_0 -t}} B \vp \,dx\,dt  \\
&\hsp - C R^{1-n} \int_{T_R(t_0)}|\N\lap u|^2 \,dx\,dt - C  R^{3-n}  \int_{\AR^n}|\lap u_0|^2 \,dx\\
&\hsp - C R^{1-n} \sup_{t\in[t_0-16 R^4, t_0-R^4]}\int_{[\varphi>0]}|\N u|^2 (\cdot, t) \,dx
\end{align*}
and 
\begin{align} \label{inequ:star} 
\tfrac{\del}{\del R} & \brk{\Psi \prs{u,R,\vp}} 
 \geq \tfrac{1}{2}\tfrac{1}{ R} \int_{T_R (t_0)} \tfrac{ \brs{\Sol u}^2}{4 \brs{t_0-t}} B \vp \,dx\,dt \\
\notag &\hsp - C R^{5-n}  \sup_{t\in[t_0-16 R^4, t_0-R^4]}\int_{[\varphi>0]}|\N\lap u|^2 \,dx - C R^{3-n}  \int_{\AR^n}|\lap u_0|^2 \,dx\\
\notag &\hsp - C R^{1-n} \sup_{t\in[t_0-16 R^4, t_0-R^4]} \int_{[\varphi>0]}|\N u|^2 (\cdot, t) \,dx.
\end{align}

 \end{thm}

 \begin{proof} We recall that $\Delta u$ is the Euclidean Laplace operator of $u$ seeing $u$ as a map from $\AR^n$ to $\AR^k$. By translation it suffices to consider the case $(x_0,t_0) = (0,0)$. We follow the proof of Theorem~\ref{mainInequ-0-withoutSUP} line by line until we arrive at (\ref{equ:halfway}) which is
\begin{align*}
	& \frac{\partial}{\partial R}  \Psi(u,R,\varphi) = \tfrac{1}{R} \int_{T_R} \langle \Delta^2 u ,   (\Sol u) \rangle B\,\varphi\, dx\, dt   \\
	\notag &\quad	+  \tfrac{2}{R}\int_{T_R} \langle\nabla\Delta u \cdot \nabla B ,   (\Sol u)\rangle \varphi \, dx\, dt  +   \tfrac{2}{R} \int_{T_R} \langle\nabla\Delta u \cdot \nabla \varphi , (\Sol u)\rangle B  \, dx \,dt\\
\notag &\quad +  \tfrac{1}{R}\int_{T_R} \langle \Delta u,  (\Sol u) \rangle \left[\Delta B\,\varphi + 2 \nabla B\cdot\nabla\varphi + B\Delta\varphi\right] \,dx\, dt\\
\notag & \quad +  \tfrac{1}{R}\int_{T_R} \langle \Delta u, (\Sol u) \rangle \Delta B\, \varphi \, dx\, dt +  \tfrac{1}{R}\int_{T_R} \langle\nabla u \cdot \tfrac{x}{(-4t)},  (\Sol u)\rangle B \, \varphi \, dx\, dt \\
\notag & \quad +  \tfrac{1}{R} \int_{T_R}\langle \nabla u \cdot \nabla\varphi, (\Sol u) \rangle \Delta B \, dx \,dt
\end{align*}
We again use $t<0$ so that $-t = |t|$. After recalling the evolution equation $\partial_t u = -\Delta^2 u - f(u)$, we include 
 the term $ f(u) \cdot (\Sol u) B\varphi$ in the first integral. We only add a zero because $\Sol u =x\cdot \nabla u + 4t\partial_t u$ is tangential to $N$ at $u$ and $f(u)$ is $L^2$-orthogonal to the tangent space of $N$ at $u$. So we again sum the first and the $6$th term and get
\begin{align*}
 1\text{st integral} & + 6\text{th integral} =  \tfrac{1}{4R} \int_{T_R} \langle \frac{4 t \partial_ t u}{|t|}  ,   (\Sol u) \rangle B\,\varphi\, dx \,dt  \\
 &\quad \qquad\qquad\qquad\qquad\qquad + \tfrac{1}{R}\int_{T_R} \langle\nabla u \cdot \frac{x}{4|t|}  , (\Sol u)\rangle B\, \varphi \, dx\, dt\\
& =  \tfrac{1}{R} \int_{T_R} \frac{|x\cdot \nabla u + 4 t \partial_ t u|^2}{4|t|}  B\,\varphi\, dx\, dt = \tfrac{1}{R} \int_{T_R} \frac{|\Sol u|^2}{4|t|} B\,\varphi\, dx\, dt.
\end{align*}
We continue the proof as in the proof of Theorem~\ref{mainthm-derivatives-withoutSUP} and keep the term $ \int_{T_R} |\nabla\Delta u|^2 dxdt$ and $\sup_{t}\int_{[\varphi>0]}|\nabla u|^2 dx$ in the estimate. The inequality $\int_{\AR^n}|\lap u|^2(x,t) \,dx\leq \int_{\AR^n}|\lap u_0|^2 \,dx$ follows from the $L^2$-gradient flow property which is (\ref{first-1}) from Lemma~\ref{lem:apriori}. 

To derive the second inequality from the first, we simply use in addition
\begin{align*}
	\int_{T_R} |\N \Delta u|^2\,dx\,dt &= \int_{t_0-16 R^4}^{t_0 -R^4}\int_{\AR^n} |\N \Delta u|^2\,dx\,dt \\
					   &\leq c R^4 \sup_{t\in [t_0-16R^4, t_0-R^4]}\int|\N\Delta u|^2 \,dx
\end{align*}
and we used the fact that we can restrict the domain of integration to the set where $\varphi>0$.
 \end{proof}

 \begin{thm}
   Let $u:\AR^n \times [0,T) \to N\hookrightarrow \AR^k$ be a smooth solution of the extrinsic biharmonic map heat flow on $\AR^n$ with smooth initial datum \break $u_0:\AR^n \to N \hookrightarrow \AR^k$, $n\leq 4$, satisfying 
 \begin{align*}
 &\int|\nabla^l u|^2 dx <\infty \qquad \text{ for each } t\in [0,T) \text{ and } l\in\{1,2,3,4\},\\
 &\sup_{t\in[0,T)} \int|u|^2 \leq c_1^2 <\infty  \qquad  \text{ and if } n=4: \int|\nabla u|^4 <\epsilon_0,
 \end{align*}
 with $\epsilon_0$ from Lemma~\ref{lem:apriori}.   
   Then, for any $(x_0,t_0)\in\AR^n \times (0,T]$ there is a smooth cut-off function $\varphi \in C^\infty_c (\AR^n \times [0,T))$ depending on $f_n$ and a constant $C>0$ depending on $f_n$ and $c_1$, $N$ and $n$, such that for every $R$ with $0< R < t_0^{1/4}/2$, we have
      \begin{align*}
\tfrac{\del}{\del R} \brk{\Psi \prs{u,R,\vp}} 
& \geq \tfrac{1}{2}\tfrac{1}{ R} \int_{T_R(t_0)} \tfrac{ \brs{\Sol u}^2}{4 \brs{t_0-t}} B \vp \,dx\,dt \\
&\hsp- C (1+T) R^{1-n} E(0)\\
&\hsp - C R^{1-n} \int_{\AR^n}|\nabla u_0|^2 \,dx - C R^{3-n} \int_{\AR^n}|\lap u_0|^2 \,dx,
\end{align*}
where $E(0)$ depends on $\int|\Delta u_0|^2 dx.$ This implies that $\Psi$ is a pre-entropy for $n=1$. For $2\leq n\leq 3$ we get that   
         \begin{align*}
\tfrac{\del}{\del R} \brk{\Psi \prs{u,R,\vp}} 
& \geq \tfrac{1}{2}\tfrac{1}{ R} \int_{T_R(t_0)} \tfrac{ \brs{\Sol u}^2}{4 \brs{t_0-t}} B \vp \,dx\,dt \\
&\hsp- C  R^{5-n} \tilde E_1(u_0)- C R^{1-\frac{n}{2}}  \tilde E_2(u_0) \\
&\hsp - C R^{3-n} \int_{\AR^n}|\lap u_0|^2 \,dx,
\end{align*}
where $\tilde E_1(u_0), \tilde E_2(u_0)$ depend on $u_0, N, n$ and $T$.
 \end{thm}

 \begin{proof}
The term involving $\nabla\Delta u$ is either estimated by Corollary~\ref{cor:W3est} if it is still integrated in time, or by Proposition~\ref{prop:absch} to get the second inequality. In the first inequality, we use Lemma~\ref{lem:apriori} (\ref{third-3}) to bound \\ $\sup_{t}\int|\nabla u|^2 \leq C \int|\nabla u_0|^2$. \\
For the second inequality Lemma~\ref{lem:apriori}
(\ref{second-zw-2u3}) provides bounds on $\int|\nabla u|^4$ depending on $c_1$ and $\int|\Delta u_0|^2$. We use this to bound the term $\sup_{t\in[t_0-16R^4, t_0-R^4]}\int|\nabla u|^2 dx$ by H\"older inequality as in Theorem~\ref{mainthm-derivatives}. The property $\operatorname{spt}(\phi) \subset B_{\rho R}(0)$ again gives us another $R^{\frac{n}{2}}$. 
 \end{proof}

We now formulate the integrated statement using Proposition~\ref{prop:absch}. 
 \begin{cor} \label{cor:secondlast}
   Let $u:\AR^n \times [0,T] \to N\hookrightarrow \AR^k$ be a smooth solution of the extrinsic biharmonic map heat flow on $\AR^n$ with smooth initial datum $u_0:\AR^n \to N \hookrightarrow \AR^k$, $1\leq n\leq 3$.
   Then, for any $(x_0,t_0)\in\AR^n \times (0,T]$ there is a smooth cut-off function $\varphi \in C^\infty_c (\AR^n \times[0,T))$ depending on $f_n$ and a constant $C>0$ depending on $f_n$, $N$, $n$, such that for  \\ $0< R_1\leq R_2 \leq  \min\{t_0^{1/4}/2,1\}$, we have
   \begin{align*}
   \Psi(u, R_1, \varphi) & \leq \Psi(u, R_2, \varphi) + C \left(R_2^{\frac{4- n}{2}}- R_1^{\frac{4-n}{2}}\right)\tilde E_0 \\
   & + C(R_2^{4-n}-R_1^{4-n}) \int_{\AR^n}|\Delta u_0|^2,
   \end{align*}
 where $\tilde E_0$ depends on $u_0,n$, $N$ and $T$. This implies that $\Psi$ is a pre-entropy in our sense for $1\leq n\leq3$.
 \end{cor}

 Also for $n=4$, we can use the results from \cite{LammExtr} to get a similar monotonicity formula.
\begin{cor}
 In the case $n=4$, let $u:\AR^4 \times [0,T) \to N\hookrightarrow \AR^k$, be a smooth solution of the extrinsic biharmonic map heat flow on $\AR^4$ with smooth initial datum $u_0:\AR^4 \to N \hookrightarrow \AR^k$. Then there is an $\epsilon_0>0$ and an $r>0$ such that if 
 \[
 \sup_{(x,t)\in \AR^4\times [0,T)} \left[ 2\int_{B_{2r}(x)}|\Delta u|^2 dx +\left(\int_{B_{2r}(x)}|\nabla u|^4 dx\right)^{\frac{1}{2}}\right] <\epsilon_0,
 \]
   then, for any $(x_0,t_0)\in\AR^n \times (0,T]$ there is a smooth cut-off function \\ $\varphi \in C^\infty_c (\AR^n \times [0,T))$ depending on  $f_4$ and a constant $C$ depending on $f_4$ such that, for  $0< R_1<R_2 < \min\{t_0^{1/4}/2,1\}$ we have
   \begin{align*}
   \Psi(u, R_1, \varphi) & \leq \Psi(u, R_2, \varphi) + C (R_2- R_1)\tilde E_0\\
   &\hsp + C (\log(R_2)- \log(R_1))\int|\Delta u_0|^2  + C (\log(R_2)-\log(R_1)) \tilde E_1,
   \end{align*}
   where $\tilde E_0, \tilde E_1$ depend on $u_0$, $N$ and $T$.
\end{cor}

\begin{proof} As in the proof of Proposition~\ref{prop:absch} we modify the results from  \cite{LammExtr} for $n=4$ to our setting. This particularly yields the existence of the claimed $\epsilon_0$ and $r$ such that, if the smallness condition along the flow is satisfied, then there is a constant $C(u_0)$ independent of $R$ such that 
 \[
 \sup_{t\in [0,t_0-R_0^4]} \int_{\AR^n}|\nabla\Delta u|^2 \sqrt{\varphi}\, dx + \sup_{t\in [0,t_0-R^4_0]} \int_{\AR^n}|\nabla u|^2 \sqrt{\varphi}\, dx \leq C(u_0).
 \]
 The power $R^{1-\frac{n}{2}} = R^{-1}$ integrates to the logarithm for $n=4$.
\end{proof}

\begin{rmk}
We saw in (\ref{mainInequ-0-forPhi}) that $\Phi$ satisfies an analogous inequality as $\Psi$ does in the case where we estimated
\[- \int_{T_R}|\nabla\Delta u|^2 dxdt\geq -  R^4\sup_t  \int_{\mathbb R^n}|\nabla \Delta u|^2 dx.\]
This implies that inequality (\ref{inequ:star}) in Theorem~\ref{mainthm-generalN} and the results from Corollary~\ref{cor:secondlast} and Corollary~\ref{cor:last} also hold true for $\Phi$ replaces by $\Psi$, where the first term in (\ref{inequ:star}) is $\frac{1}{8 R}\int_{S_R(t_0)} |\Sol u|^2 B \varphi dx$ instead of $\frac{1}{2 R}\int_{T_R(t_0)}\frac{|\Sol u|^2}{4 |t_0-t|} B\varphi dx dt$.
\end{rmk}

\bibliographystyle{plain}
\bibliography{sources.bib}

\end{document}